\documentclass[12pt,twoside]{amsart}
\usepackage{amsmath, amsthm, amscd, amsfonts, amssymb, graphicx}
\usepackage[bookmarksnumbered, plainpages]{hyperref}

\newcommand{\tr}{\text{tr}}

\newtheorem{thm}{Theorem}[section]

\newtheorem{lem}[thm]{Lemma}
\newtheorem{prop}[thm]{Proposition}
\newtheorem{defn}[thm]{Definition}

\numberwithin{equation}{section}

\begin{document}

\title{\bf Flat Higgs bundles  over non-compact affine Gauduchon manifolds}
\author{Zhenghan Shen, Chuanjing Zhang and Xi Zhang}

\address{Zhenghan Shen\\School of Mathematical Sciences\\
University of Science and Technology of China\\
Hefei, 230026,P.R. China\\}\email{szh2015@mail.ustc.edu.cn}

\address{Chuanjing Zhang\\School of Mathematical Sciences\\
University of Science and Technology of China\\
Hefei, 230026,P.R. China\\}\email{chjzhang@mail.ustc.edu.cn}
\address{Xi Zhang\\School of Mathematical Sciences\\
University of Science and Technology of China\\
Hefei, 230026,P.R. China\\ } \email{mathzx@ustc.edu.cn}

\subjclass[]{53C07, 57N16}
\keywords{Affine Gauduchon manifolds,\ Flat Higgs bundles,\ Affine Hermitian-Yang-Mills flow,\ Affine Hermitian-Einstein metric }

\maketitle

\begin{abstract}
In this paper, we use the affine Hermitian-Yang-Mills flow to prove a generalized Donaldson-Uhlenbeck-Yau theorem on flat Higgs bundles over a class of non-compact affine Gauduchon manifolds.
\end{abstract}

\vskip 0.2 true cm

%------------------------------------------------------------------------------------%

\pagestyle{myheadings}
\markboth{\rightline {\scriptsize Z. Shen et al.}}
         {\leftline{\scriptsize Flat Higgs bundles  over non-compact affine Gauduchon manifolds}}

\bigskip
\bigskip

%------------------------------------------------------------------------------------%
%------------------------------------------------------------------------------------%

\section{ Introduction}

An affine manifold is a smooth real manifold $X$ equipped with a flat torsion-free connection $D$ on its tangent bundle $TX$, or equivalently an affine structure which is provided by an atlas of $X$ whose transition functions are affine maps of the form
\begin{equation}
  x\mapsto Ax+b,
\end{equation}
where $A\in {\rm GL}(n,\mathbb{R})$ and $b\in \mathbb{R}^n$. An affine manifold $X$ is called special if it admits a volume form $\nu$ which is covariant constant with respect to the flat connection $D$, or  equivalently, the holonomy group of $D$ lies in the special affine subgroup, i.e. $A \in {\rm SL}(n,\mathbb{R}) $. Special affine manifolds play a role in the Strominger-Yau-Zaslow conjecture (\cite{StroYau96}). Cheng and Yau (\cite{CY82}) proved that on a closed special affine manifold, an affine K\"ahler metric (if it exists) can be deformed to a flat metric by adding the Hessian of a smooth function. This result shows that  the compact special affine manifolds are all torus quotients. A Riemannian metric $g$ on an affine manifold is said to be affine Gauduchon if
\begin{equation}
  \partial\bar{\partial}(\omega_g^{n-1})=0,
\end{equation}
where $\partial$,$\bar{\partial}$ and $\omega_g$ are defined in section \ref{Pre_res}. On a compact special affine manifold $M$, every conformal class of Riemannian metrics contains an affine Gauduchon metric, which is unique up to scaling by a constant (see \cite{Loftin09}).

%Let $(X,\omega)$ be a compact K\"ahler manifold and $(E,\bar{\partial}_E,\theta)$ be a Higgs bundle over $X$. That is, $(E,\bar{\partial}_E)$ is a holomorphic vector bundle coupled with a Higgs filed $\theta \in \Omega_X^{1,0}({\rm End}(E))$ such that $\bar{\partial}_E\theta=0$ and $\theta \wedge \theta =0$. The Higgs bundle $(E,\bar{\partial}_E,\theta)$ is called stable if for every proper saturated $\theta$-invariant sub-Higgs sheaf $F\subset E$,
%\begin{equation}
%\frac{{\rm deg}_{\omega}(F)}{{\rm rank}(F)}<\frac{{\rm deg}_{\omega}(E)}{{\rm rank}(E)},
%\end{equation}
%where ${\rm deg}_{\omega}(E)$ is the $\omega$-degree of $E$ given by
%\begin{equation}
%{\rm deg}_{\omega}(E)=\int_Xc_{1}(E)\wedge \frac{\omega^{n-1}}{(n-1)!}.
%\end{equation}
%
%Let $H$ be a Hermitian metric on the bundle $E$. We say $H$ is a Hermitian-Einstein metric on Higgs bundle $(E,\bar{\partial}_E,\theta)$ if the curvature of the Hitchin-Simpson connection satisfies the Einstein condition, i.e.
%\begin{equation}
%\sqrt{-1}\Lambda_{\omega}(F_H+[\theta,\theta^{*H}])=\lambda\cdot {\rm Id}_E,
%\end{equation}
%where $\Lambda_{\omega}$ denotes the contraction with $\omega$ and $\lambda$ is a constant.

A Higgs bundle $(E,\bar{\partial}_E,\theta)$ is a holomorphic vector bundle $E$ coupled with a Higgs filed $\theta \in \Omega_X^{1,0}({\rm End}(E))$ such that $\bar{\partial}_E\theta=0$ and $\theta \wedge \theta =0$. Higgs bundles were introduced by Hitchin (\cite{Hitchin87}) in his study of the self duality equations. They have rich structures and play an important role in many areas including gauge theory, K\"ahler and hyper-K\"ahler geometry, group representations and nonabelian Hodge theory. When $(X,\omega)$ is a compact K\"ahler manifold, Hitchin (\cite{Hitchin87}) and Simpson (\cite{Sim88}) obtained a Higgs bundle version of Donaldson-Uhlenbeck-Yau theorem (\cite{NS65}, \cite{Do85}, \cite{UY86}, \cite{Do87}), i.e. they proved that a Higgs bundle admits the Hermitian-Einstein metric if and only if it's Higgs poly-stable. Simpson (\cite{Sim88}) also considered some non-compact K\"ahler manifolds case. He introduced the concept of analytic stability for Higgs bundles and proved that the analytic stability implies the existence of Hermitian-Einstein metric. In \cite{ZZZ}, the authors also proved a generalized Donaldson-Uhlenbeck-Yau theorem on Higgs bundles over a class of non-compact Gauduchon manifolds.

When the base space $(X,D,g,\nu)$ is a compact special affine manifold equipped with an affine Gauduchon metric $g$, Loftin (\cite{Loftin09}) established a Donaldson-Uhlenbeck-Yau type theorem. He proved that if a flat complex vector bundle $E$ over $(X,D,g,\nu)$ is stable, then there is an affine Hermitian-Einstein metric on $E$. Biswas, Loftin and Stemmler (\cite{BLS13}) also studied the flat Higgs bundle case. The definition of the flat Higgs bundle will be introduced in section \ref{Pre_res}. They proved  that a flat Higgs bundle over a compact special affine Gauduchon manifold  is poly-stable if and only if it admits an affine Hermitian-Einstein metric. There are many other interesting and important works related (\cite{ConsulPrada03}, \cite{Biquard96}, \cite{BisSch09}, \cite{JostZuo96}, \cite{JN99}, \cite{LZZ17}, \cite{LZZ18}, \cite{LZ15},
\cite{Moch06}, \cite{Moch09}, \cite{WZ11}).

In this paper, we study the non-compact and affine Gauduchon case. In the sequel, we always assume that $(X,D,g,\nu)$ is a special affine Gauduchon manifold unless otherwise stated. By Simpson(\cite{Sim88}), we will make the following three assumptions:

{\bf Assumption 1.} $(X,D,g,\nu)$ has finite volume.

{\bf Assumption 2.} There exists a non-negative exhaustion function $\phi$ with ${\rm tr}_g\partial \bar{\partial}\phi$ bounded.

{\bf Assumption 3.} There is an increasing function $a:[0,+\infty)\rightarrow [0,+\infty)$ with $a(0)=0$ and $a(x)=x$ for $x>1$ such that if $f$ is a bounded positive function on $X$ with ${\rm tr}_g\partial\bar{\partial}f\geq -B$, then
\begin{equation}
\sup\limits_X|f|\leq C(B)a\left(\int_X|f|\frac{\omega^n_g}{\nu}\right),
\end{equation}
where $B$ is a constant. Furthermore, if ${\rm tr}_g\partial\bar{\partial}f\geq 0$, then ${\rm tr}_g\partial\bar{\partial}f=0$.

In the non-compact and affine Gauduchon case, we fix a background metric $H_0$ on the flat Higgs bundle $(E,\nabla,\theta)$, where $\nabla$ is the flat connection on $E$, and define the analytic degree of $E$ to be a real number
\begin{equation}
  \deg_g(E,H_0)=\frac{1}{n}\int_{X}{\rm tr}\left({\rm tr}_g(F_{H_0}+[\theta,\theta^{*H_0}])\right)\frac{\omega_g^n}{\nu},
\end{equation}
where $F_{H_0}$ denotes the extended curvature form which is defined in section \ref{Pre_res}.
For the flat sub-Higgs bundle $V\subset E$, we denote $F_V$ by the extended curvature form of the extended Hermitian connection on $V$ with respect to the Hermitian metric $H_0|_V$ which is the restriction of $H_0$ to $V$. According to Chern-Weil formula, we can define the analytic degree of every $\theta$-invariant flat sub-Higgs bundle $V$ of $(E,\nabla,\theta)$ by
\begin{equation}
  \deg_g(V,H_0)=\frac{1}{n}\int_X\left({\rm tr}(\pi{\rm tr}_g(F _{H_0}+[\theta,\theta^{*H_0}]))-|D''\pi|_{H_0}^2\right)\frac{\omega^n_g}{\nu},
\end{equation}
where $D''=\bar{\partial}+\theta$ and $\pi$ denotes the projection onto $V$ with respect to the metric $H_0$. Following Simpson (\cite{Sim88}), we say that the flat Higgs bundle $(E,\nabla,\theta)$ is $H_0$-analytic stable (semi-stable) if for every proper $\theta$-invariant flat sub-Higgs bundle $V\subset E$, we have
\begin{equation}
  \frac{\deg_g(V,H_0)}{{\rm rank}(V)}<(\leq)\frac{\deg_g(E,H_0)}{{\rm rank}(E)}.
\end{equation}

\begin{defn}
  An affine Hermitian-Einstein metric on a flat Higgs bundle $(E,\nabla,\theta)$ is a Hermitian metric $H$   satisfying the equation
  \begin{equation}
   {\rm tr}_g(F_H+[\theta,\theta^{*H}])=\lambda\cdot {\rm Id}_E,
  \end{equation}
  for some constant scalar $\lambda$, which is called the Einstein factor.
\end{defn}

In this paper, we will show that, under some assumptions on the special affine Gauduchon manifold $(X,D,g,\nu)$, the stability of $(E,\nabla,\theta)$ implies the existence of affine Hermitian-Einstein metric on it, i.e.  we obtain the following Donaldson-Uhlenbeck-Yau type theorem.
\begin{thm}\label{Maintheorem}
Let $(X,D,g,\nu)$ be a non-compact special affine Gauduchon manifold satisfying the Assumptions 1,2,3 and %$|\frac{\bar{\partial}\omega^{n-1}}{\nu}|_g\in L^2(X)$,
$|d\omega_g|_g\in L^{\infty}(X)$, $(E,\nabla,\theta)$ be a flat Higgs bundle over $X$ with a background Hermitian metric $H_0$ satisfying $\sup\limits_X|{\rm tr}_g(F_{H_0}+[\theta,\theta^{*H_0}])|_{H_0}<+\infty$. If $(E,\nabla,\theta)$ is $H_0$-analytic stable, then there exists a metric $H$ with $D''(\log H_0^{-1}H)\in L^2$, $H$ and $H_0$  mutually bounded, such that
\begin{equation}
{\rm tr}_g(F_H+[\theta,\theta^{*H}])=\lambda \cdot {\rm Id}_E,
\end{equation}
where the constant $\lambda=\frac{{\rm deg}_{g}(E,H_0)}{{\rm rank}(E){\rm Vol}(X)}$.
\end{thm}

We should remark that Biswas, Loftin and Stemmler (\cite{BLS13}) only proved the Donaldson-Uhlenbeck-Yau type theorem in compact case. They used the continuity method and followed the argument of L\"ubke and Teleman (\cite{LT95}). Here we first introduce the  affine Hermitian-Yang-Mills flow
\begin{equation}
  H^{-1}(t)\frac{\partial H(t)}{\partial t}=-4\left({\rm tr}_g(F_{H(t)}+[\theta,\theta^{*H(t)}])-\lambda\cdot {\rm Id}_E\right),
\end{equation}
on Higgs bundles over affine Gauduchon manifolds and then we study the existence of short-time solutions and long-time solutions. In \cite{Sim92},  Simpson used Donaldson's heat flow method to attack the existence problem of the Hermitian-Einstein metrics on Higgs bundles, and his proof relies on the properties of the Donaldson functional. However, the Donaldson functional is not well defined when $g$ is only affine Gauduchon. So Simpson's argument is not applicable in our situation directly.  In this paper, we study the non-compact affine Gauduchon cases by using the heat flow method and avoid the Donaldson functional. The key is that we use the important identity (\ref{Keyequality}) instead of the Donaldson functional. For simplicity, we set
\begin{equation}
\Phi(H,\theta)={\rm tr}_g(F_H+[\theta,\theta^{*H}])-\lambda \cdot {\rm  Id}_E.
\end{equation}
Under the assumptions as that in Theorem \ref{Maintheorem}, we can prove the following identity:
\begin{equation}\label{Keyequality}
\int_X{\rm tr}\left(\Phi(H_0,\theta)s\right)+\langle\Psi(s)(D''s),D''s\rangle_{H_0}\frac{\omega^n_g}{\nu}=\int_X{\rm tr}\left(\Phi(H,\theta)s\right)\frac{\omega^n_g}{\nu},
\end{equation}
where $s=\log(H_0^{-1}H)$ and
\begin{equation}\label{Psixy}
\Psi(x,y)=\left\{\begin{split}
  &\frac{e^{y-x}-1}{y-x}, \ \ &x\neq y,\\
  &1, \ \  &x=y.
\end{split}\right.
\end{equation}
By the above identity (\ref{Keyequality}) and Loftin's result (\cite{Loftin09}) that $L_1^2 $ weakly flat sub-bundles are smooth flat sub-bundles, using the heat flow method, we can obtain the existence result of affine Hermitian-Einstein metric, see section 5 for details. In \cite{Jacob11}, Jacob used the heat flow technique for Higgs bundle over compact Gauduchon manifold. It should be pointed out that our argument is new even in the compact Hermitian case.

This paper is organized as follows. In section 2, we give some introduction of the basic theory of $(p,q)$-forms with values in a flat vector bundle, affine Hermitian-Einstein metric and flat Higgs bundle. In section 3, we introduce the affine Hermitian-Yang-Mills flow on Higgs bundles over affine Gauduchon manifolds and give some basic estimates of the flow which will be used in section 4. Then we study the existence of short-time solutions and long-time solutions when the base space  is compact. And we also solve the Dirichlet problem for the Hermitian-Einstein equation on a flat Higgs bundle. In section 4, we get the existence of the long-time solutions of the affine Hermitian-Yang-Mills flow when the base manifold is noncompact. In section 5, we  consider the stable case and complete the proof of Theorem \ref{Maintheorem}.

\medskip

{\bf  Acknowledgement:} The authors are partially supported by NSF in China No.11625106, 11571332 and 11721101. The second author is also supported by  NSF in China No.11801535, the China Postdoctoral Science Foundation (No.2018M642515) and the Fundamental Research Funds for the Central Universities.

\medskip

\section{Preliminary}\label{Pre_res}
\subsection{Affine Dolbeault complex}
Let $(M,D)$ be an affine manifold of dimension $n$, meaning that $D$ is a flat torsion-free connection on the tangent bundle $TM$. The affine structure of $M$ determines a natural complex structure of $TM$. In local affine coordinates $\{x^1,\cdots,x^n\}$, the fact that every tangent vector $Y\in TM$ can be written as $Y=y^i\frac{\partial}{\partial x^i}$ induces the local coordinates $\{y^1,\cdots\,y^n\}$ on $TM$. Then $z^i=x^i+\sqrt{-1}y^i$ form the complex coordinates on $TM$, which  is usually denoted by $M^{\mathbb{C}}$. On $M$, there are natural $(p,q)$-forms (see \cite{Shima}), which are the restriction of $(p,q)$-forms from $M^{\mathbb{C}}$. Let $T^*M$ be the cotangent bundle over $M$. We denote by $(
\Lambda^pT^*M)\otimes (\Lambda^qT^*M)$ the tensor product of vector bundles $\Lambda^p(T^*M)$ and $\Lambda^q(T^*M)$, and by $\mathcal{A}^{p,q}$ the space of all smooth sections of $\Lambda^pT^*M\otimes \Lambda^qT^*M$.

If $\{x^i\}_{i=1}^n$ are the local affine coordinates on $M$ with respect to the connection $D$, then we will denote the induced frame on $\mathcal{A}^{p,q}$ as
\begin{equation*}
  \{dz^{i_1}\wedge\cdots\wedge dz^{i_p}\otimes d\bar{z}^{j_1}\wedge\cdots\wedge d\bar{z}^{j_q}\},
\end{equation*}
where $z^i=x^i+\sqrt{-1}y^i$ are the complex coordinates on $M^{\mathbb{C}}$. Note that $dz^i=d\bar{z}^i=dx^i$ on $M$.

There is a natural restriction map from $\Lambda^{p,q}(M^{\mathbb{C}})$ to $\mathcal{A}^{p,q}$ as follows:
for $\psi \in \Lambda^{p,q}({M^{\mathbb{C}}})$, in local affine coordinates on an open subset $U\subset M$, we have
\begin{equation}
\begin{split}
  &\sum \psi_{i_1\cdots i_p,j_1\cdots j_q}(dz^{i_1}\wedge\cdots\wedge  dz^{i_p})\wedge(d\bar{z}^{j_1}\wedge\cdots\wedge d\bar{z}^{j_q})\\
  &\mapsto \sum \psi_{i_1\cdots i_p,j_1\cdots j_q}|_U(dz^{i_1}\wedge\cdots\wedge dz^{i_p})\otimes(d\bar{z}^{j_1}\wedge\cdots\wedge d\bar{z}^{j_q}),
\end{split}
\end{equation}
where $\psi_{i_1\cdots i_p,j_1\cdots j_q}$ are smooth functions on $TU\subset TM=M^{\mathbb{C}}$, $U$ is considered as the zero section of $\pi:TU \rightarrow U$, that is for a function $f$ on $M^{\mathbb{C}}$, $f(x,y)|_U=f(x,0)$. And the sums are taken over all $1\leq i_1<\cdots<i_p\leq n$ and $1\leq j_1<\cdots<j_q\leq n$.

We can also define the natural operators $\partial$ and $\bar{\partial}$ acting on $(p,q)$-forms on $M$
\begin{equation*}
\partial :\mathcal{A}^{p,q}\rightarrow \mathcal{A}^{p+1,q} \quad \text{and}\quad \bar{\partial}:\mathcal{A}^{p,q}\rightarrow \mathcal{A}^{p,q+1}
\end{equation*}
by
\begin{equation}
\partial :=\frac{1}{2}(d^{D}\otimes {\rm Id})\quad\text{and}\quad \bar{\partial}:=(-1)^p\frac{1}{2}({\rm Id}\otimes d^D),
\end{equation}
where $d^D$ is the exterior derivative for the forms of $M$ induced from $D$ and ${\rm Id}$ is the identity operator. In local affine coordinates, for $\phi=\phi_{i_1\cdots i_p,j_1\cdots j_q}(dz^{i_1}\wedge\cdots \wedge dz^{i_p})\otimes (d\bar{z}^{j_1}\wedge\cdots\wedge d\bar{z}^{j_q})\in \mathcal{A}^{p,q}$,
\begin{equation}
\partial \phi:=\frac{1}{2}(d\phi_{i_1\cdots i_p,j_1\cdots j_q})\wedge (dz^{i_1}\wedge\cdots \wedge dz^{i_p})\otimes (d\bar{z}^{j_1}\wedge\cdots\wedge d\bar{z}^{j_q})
\end{equation}
and
\begin{equation}
\bar{\partial}\phi:=(-1)^p\frac{1}{2}(dz^{i_1}\wedge\cdots\wedge dz^{i_p})\otimes (d\phi_{i_1\cdots i_p,j_1\cdots j_q})\wedge(d\bar{z}^{j_1}\wedge\cdots\wedge d\bar{z}^{j_q}).
\end{equation}
These operators are the restrictions of the corresponding $\partial$ and $\bar{\partial}$ operators on $M^{\mathbb{C}}$ with respect to the restriction map above.

There is also a natural wedge product on $\mathcal{A}^{p,q}$, which we take to be the restriction of the wedge product on $M^\mathbb{C}$: if $\phi_i\otimes \psi_i \in \mathcal{A}^{p_i,q_i}$ for $i=1,2$, then we define
\begin{equation}
(\phi_1\otimes\psi_1)\wedge(\phi_2\otimes\psi_2):=(-1)^{q_1p_2}(\phi_1\wedge\phi_2)\otimes (\psi_1\wedge \psi_2)\in \mathcal{A}^{p_1+p_2,q_1+q_2}.
\end{equation}

We can also define the conjugation operator on $(p,q)$-forms on $M$, which is the restriction of the complex conjugation on $M^{\mathbb{C}}$: if $\phi \in \Lambda^p(T^*M)$ and $\psi \in \Lambda^q(T^*M)$ are complex valued forms, then we define
\begin{equation}
\overline{\phi\otimes\psi}:=(-1)^{pq}\bar{\psi}\otimes\bar{\phi}.
\end{equation}

Given a smooth Riemannian metric $g$ on an affine manifold $M$, there is a natural nondegenerate $(1,1)$-form $\omega_g$ on $M$ expressed in local affine coordinate as
\begin{equation}
\omega_g=\sum\limits_{i,j=1}^ng_{ij}dz^i\otimes d\bar{z}^j.
\end{equation}
The Riemannian metric $g$ naturally is extended to a Hermitian metric on $M^{\mathbb{C}}$ and $\omega_g$ is the restriction of the corresponding $(1,1)$-form on $M^{\mathbb{C}}$.
%A Riemannian metric $g$ on $M$ is said to be affine Gauduchon if
%\begin{equation}
%\partial\bar{\partial}(\omega_g^{n-1})=0,
%\end{equation}
%where $\partial$ and $\bar{\partial}$ are operators on $M$ defined above. On a compact special affine manifold $M$, every conformal class of Riemannian metrics contains an affine Gauduchon metric, which is unique up to scaling by a constant (see \cite{Loftin09}).

%The affine manifold $M$ is called special if it admits a volume form $\nu$ which is covariant constant with respect to the flat connection $D$ on $TM$.

In the case of special affine structures, $\nu$ induces the natural maps
\begin{equation}
\mathcal{A}^{n,q}\rightarrow \Lambda^q(T^*M),\quad \nu\otimes \chi \mapsto (-1)^{\frac{n(n-1)}{2}}\chi
\end{equation}
and
\begin{equation}
\mathcal{A}^{p,n}\rightarrow \Lambda^p(T^*M),\quad \chi\otimes \nu \mapsto (-1)^{\frac{n(n-1)}{2}}\chi,
\end{equation}
which are called division by $\nu$. In particular, any $(n,n)$-form $\chi$ can be integrated by considering the integral of $\frac{\chi}{\nu}$(see \cite{Loftin09}).

Let $(E,\nabla)$ be a flat complex vector bundle of rank $r$ over an affine manifold $M$. As we can see, the pull back of $E$ to $TM=M^{\mathbb{C}}$ by the natural projection $\pi:TM\rightarrow M$ will be denoted by $E^{\mathbb{C}}$. The flat connection $\nabla$ is pulled back to a flat connection $\pi^*\nabla$ on $E^{\mathbb{C}}$. This flat vector bundle $(E^{\mathbb{C}},\pi^*\nabla)$ over $M^{\mathbb{C}}$  can be considered as extension of the flat vector bundle $(E,\nabla)$ on the zero section of $TM$.

Let $\{s_1,\cdots, s_r\}$ be the locally constant frames on $E$ with respect to the flat connection $\nabla$, meaning that $\nabla(s_{\alpha})=0$, for $\alpha=1,2,\cdots ,r$. Any locally constant frame $\{s_1,\cdots,s_r\}$ of $E$ over $M$ can be extended to a locally constant frame of $E^{\mathbb{C}}$ over $M^{\mathbb{C}}$, and they are also holomorphic frames on $E^{\mathbb{C}}\rightarrow M^{\mathbb{C}}$ with respect to the holomorphic structure of $E^{\mathbb{C}}$ induced by $\nabla$.

Let $H$ be a Hermitian metric on $E$. It defines a Hermitian metric on $E^{\mathbb{C}}$ and let $D_H$ be the Chern connection associated to this Hermitian metric on $E^{\mathbb{C}}$. Then according to the decomposition of $(1,0)$-part and $(0,1)$-part, $D_H$ corresponds to a pair
\begin{equation}
(\partial_H,\bar{\partial})=(\partial_{H,\nabla},\bar{\partial}_{\nabla}),
\end{equation}
where
\begin{equation*}
\partial_{H,\nabla}:\Gamma(E)\rightarrow \mathcal{A}^{1,0}(E)\quad \text{and}\quad \bar{\partial}_{\nabla}: \Gamma(E)\rightarrow \mathcal{A}^{0,1}(E)
\end{equation*}
are smooth differential operators defined as that in \cite{Loftin09}. The pair $(\partial_H,\bar{\partial})$ is called the extended Hermitian connection of $(E,H)$.

For locally constant frame $\{s_1,\cdots,s_r\}$, if we write $H_{\alpha\bar{\beta}}=H(s_{\alpha},s_{\beta})$, then we can locally define the extended connection form
\begin{equation}
A_H=H^{-1}\partial H\in \mathcal{A}^{1,0}({\rm End} E),
\end{equation}
the extended curvature form
\begin{equation}
F_H=\bar{\partial}A_H=\bar{\partial}(H^{-1}\partial H)\in \mathcal{A}^{1,1}({\rm End} E),
\end{equation}
the extended mean curvature
\begin{equation}
K_H={\rm tr}_gF_H\in C^{\infty}(M,{\rm End} E),
\end{equation}
and the extended first Chern form
\begin{equation}
c_1(E,H)={\rm tr} F_H \in \mathcal{A}^{1,1},
\end{equation}
which are the restrictions of the corresponding objects on $E^{\mathbb{C}}$. Here ${\rm tr}_g$ denotes the contraction of differential forms using Riemannian metric $g$, and ${\rm tr}$ denotes the trace map on the fibers of ${\rm End} E$. Note that the extended first Chern form is given locally by
\begin{equation}
c_1(E,H)=-\partial\bar{\partial}(\log\det(H_{\alpha\bar{\beta}})).
\end{equation}

%\begin{defn}
%  A Hermitian metric $H$ on $E$ is called an affine Hermitian-Einstein metric with respect to $g$ if the extended mean curvature $K_H$ is of the form
%  \begin{equation}
%  K_H=\lambda\cdot {\rm Id}_{E},
%  \end{equation}
%  for some constant $\lambda$.
%\end{defn}

The degree of $(E,\nabla)$ with respect to an affine Gauduchon metric $g$ on $M$ is defined to be
\begin{equation}
{\rm deg}_g(E):=\int_M\frac{c_1(E,H)\wedge \omega_g^{n-1}}{\nu},
\end{equation}
where $H$ is any Hermitian metric on $E$. Since $c_1(E,H)=c_1(E,K)+\partial\bar{\partial}(\log\frac{\det(K_{\alpha\bar{\beta}})}{\det(H_{\alpha\bar{\beta}})}) $, this is well defined by \cite{Loftin09} when $M$ is compact.

%As usual, if ${\rm rank}(E)>0$, the slope of $E$ with respect to $g$ is defined to be
%\begin{equation}
%\mu_g(E):=\frac{{\rm deg}_g(E)}{{\rm rank} (E)}.
%\end{equation}
%
%\begin{defn}
%  The flat vector bundle $(E,\nabla)$ is called stable (respectively, semi-stable) with respect to $g$ if  for every proper non-zero flat sub-bundle $F$ of $E$, we have
%  \begin{equation}
%  \mu_g(F)<\mu_g(E)\quad (\text{respectively}\quad \mu_g(F)\leq \mu_g(E)).
%  \end{equation}
%\end{defn}

\vskip 1 true cm

\subsection{Flat Higgs vector bundles}

\begin{defn}
  Let $(E,\nabla)$ be a smooth vector bundle $E$ equipped with a flat connection $\nabla$ over an affine manifold $M$. A flat Higgs field $\theta$ on $(E,\nabla)$ is defined to be a smooth section of $T^*M\otimes {\rm End}(E)$ such that
  \begin{enumerate}
    \item[(1)] $\theta$ is covariant constant, that is $d^{\nabla}(\theta)=0$, where
    \begin{equation*}
    d^{\nabla}:\Gamma(T^*M\otimes {\rm End} E)\longrightarrow \Gamma\big((\wedge ^2T^*M)\otimes {\rm End} E\big)
    \end{equation*}
    is the induced covariant derivation.
    \item[(2)]$\theta\wedge \theta=0$.
  \end{enumerate}
  If $\theta$ is a flat Higgs field on $(E,\nabla)$, then $(E,\nabla,\theta)$ is called a flat Higgs bundle.
\end{defn}

A Higgs field $\theta$ will always be understood as an element of $\mathcal{A}^{1,0}({\rm End}E)$, meaning it is expressed in local affine coordinates as
\begin{equation*}
\theta=\sum\limits_{i=1}^{n}\theta_i\otimes dz^i,
\end{equation*}
where $\theta_i$ are locally defined flat sections of ${\rm End}(E)$ and note that $dz^i=dx^i$ on $M$.

Given a Hermitian metric $H$  on $(E,\theta)$, the adjoint $\theta^{*H}$ of $\theta$ with respect to $H$ will be regarded as an element of $\mathcal{A}^{0,1}({\rm End} E)$. In locally affine coordinates, this means that
\begin{equation*}
\theta^{*H}=\sum\limits_{j=1}^n(\theta_j)^{*H}\otimes d{\bar{z}}^j.
\end{equation*}

In particular, the Lie bracket $[\theta,\theta^{*H}] \in \mathcal{A}^{1,1}({\rm End}E)$ is locally written as
\begin{equation*}
[\theta,\theta^{*H}]=\theta\wedge \theta^{*H}+\theta^{*H}\wedge \theta=\sum\limits_{i,j=1}^n\left(\theta_i\circ(\theta_j)^{*H}-(\theta_j)^{*H}\circ\theta_i\right)\otimes dz^i\otimes d\bar{z}^j.
\end{equation*}

The extended connection form $A_{H,\theta}$ of the  flat Higgs bundle $(E,\nabla,\theta,H)$ is defined by
\begin{equation}
A_{H,\theta}:=(A_H+\theta,\theta^{*H})\in \mathcal{A}^{1,0}({\rm End}E)\oplus\mathcal{A}^{0,1}({\rm End}E).
\end{equation}
This extended connection form corresponds to the connection form of the Hitchin-Simpson connection $D_{H,\theta}=D_{H}+\theta+\theta^{*H}$ on $E^{\mathbb{C}}\rightarrow M^{\mathbb{C}}$. Analogously, the extended curvature form $F_{H,\theta}$ of $(E,\nabla,\theta,H)$ is defined to be
\begin{equation}
\begin{split}
F_{H,\theta}&:=(\partial_H\theta,\bar{\partial}A_H+[\theta,\theta^{*H}],\bar{\partial}(\theta^{*H}))\\
            &\in \mathcal{A}^{2,0}({\rm End}E)\oplus \mathcal{A}^{1,1}({\rm End} E)\oplus \mathcal{A}^{0,2}({\rm End}E).
\end{split}
\end{equation}
It corresponds to the curvature form $F_{H,\theta}=F_H+[\theta,\theta^{*H}]+\partial_{H}\theta+\bar{\partial}\theta^{*H}$ of the connection $D_{H,\theta}$ on $E^{\mathbb{C}}$. As in usual case, the extended mean curvature $K_{H,\theta}$ of $(E,\nabla,\theta,H)$ is obtained by contracting the $(1,1)$-part of the extended curvature $F_{H,\theta}$ using the Riemannian metric $g$, i.e.
\begin{equation}
K_{H,\theta}:={\rm tr}_g\left(\bar{\partial}A_H+[\theta,\theta^{*H}]\right)\in \mathcal{A}^{0,0}({\rm End}E).
\end{equation}
Since ${\rm tr}[\theta,\theta^{*H}]=0$, we have ${\rm tr}K_{H,\theta}={\rm tr}K_H$. So the extended mean curvature $K_{H,\theta}$ of $(E,\nabla,\theta,H)$ is also related to the first Chern form $c_1(E,H)$ by
\begin{equation}
\left({\rm tr}K_{H,\theta}\right)\omega_g^n=nc_1(E,H)\wedge \omega_g^{n-1}.
\end{equation}

%\begin{defn}
%  An affine Hermitian-Einstein metric on a flat Higgs bundle $(E,\nabla,\theta)$ is a Hermitian metric $H$  such that the extended mean curvature $K_{H,\theta}$ satisfies the equation
%  \begin{equation}
%   K_{H,\theta}=\lambda\cdot {\rm Id}_E,
%  \end{equation}
%  for some constant scalar $\lambda$, which is called the Einstein factor.
%\end{defn}
%
%Let $(E,\theta)$ be a flat Higgs bundle on a compact special affine manifold $(M,D,g,\nu)$ equipped with an affine Gauduchon metric $g$.
%\begin{enumerate}
%  \item[(1)]$(E,\theta)$ is called {\it stable} (respectively,{\it semistable}) if for every flat Higgs sub-bundle $F$ of $E$, with $0<{\rm rank}(F)<{\rm rank} (E)$ which is preserved by $\theta$, that is $\theta(F)\subset T^*M\otimes F$, we have
%      \begin{equation*}
%      \mu_g(F)<\mu_g(E)\qquad\left(\text{respectively},\mu_g(F)\leq\mu_g(E)\right).
%      \end{equation*}
%
% \item[(2)]$(E,\theta)$ is called {\it polystable} if
% \begin{equation}
%(E,\theta)=\bigoplus_{i=1}^N(E_i,\theta|_{E_i}),
%\end{equation}
% where all the $(E_i,\theta|_{E_i})$ are the stable flat Higgs bundles of the same slope $\mu_g(E_i)=\mu_g(E)$.
%\end{enumerate}

\section{The affine Hermitian-Yang-Mills flow on Higgs bundles over compact affine Gauduchon manifolds}

\subsection{The affine Hermitian-Yang-Mills flow}
Let $(M,D,g)$ be an $n$-dimensional compact affine  manifold equipped with an affine Gauduchon metric $g$, $(E,\nabla,\theta)$ be a rank $r$ flat Higgs bundle over $M$ with a background metric $H_0$. Consider the evolution equation
\begin{equation}\label{AHYMF}
  H^{-1}(t)\frac{\partial H(t)}{\partial t}=-4({\rm tr}_g (F_{H(t)}+[\theta,\theta^{*H(t)}])-\lambda \cdot {\rm Id}_E).
\end{equation}
We call it the affine Hermitian-Yang-Mills flow. Choosing a local affine coordinates system $\{x^i\}_{i=1}^n$ on $M$, we define the affine Laplace operator for a function $f$ as
\begin{equation}
\tilde{\Delta}f=4{\rm tr}_g\partial\bar{\partial}f=g^{ij}\frac{\partial ^2f}{\partial x^i\partial x^j},
\end{equation}
where $(g^{ij})$ is the inverse matrix of the matrix $(g_{ij})$. As usual, we denote the Beltrami Laplacian operator by
\begin{equation}
\Delta f=g^{ij}\frac{\partial ^2 f}{\partial x^i\partial x^j}+\left[\frac{\partial g^{ij}}{\partial x^i}+g^{kj}\Gamma^i_{ik}\right]\frac{\partial f}{\partial x^j}.
\end{equation}
The difference between the two Laplacians is given by a first order differential operator as follows
\begin{equation}
(\tilde{\Delta}-\Delta)f=\langle V,\nabla f\rangle_g,
\end{equation}
where $V$ is a well-defined vector field on $M$, which is locally expressed as $V=V_l\frac{\partial}{\partial x^l}=(\frac{\partial g^{il}}{\partial x^i}+g^{kl}\Gamma^i_{ik})\frac{\partial}{\partial x^l}$.

Take a locally constant frame $s_{\alpha}(1\leq\alpha\leq r)$ on $E$, and express $H(t), F_{H(t)}, \theta, \theta^{*H(t)}$ locally. When there is no confusion, we will omit parameter $t$ and simply write $H$, $F_H$, $h$ for $H(t)$, $F_{H(t)}$, $h(t)$, respectively. Then the affine Hermitian-Yang-Mills flow (\ref{AHYMF}) can be rewritten as follows:
\begin{equation}\label{NAHYMF}
\begin{split}
  \frac{\partial H(t)}{\partial t}&=-4{\rm tr}_g\bar{\partial}\partial H+4{\rm tr}_g\bar{\partial}H H^{-1}\partial H +4(\lambda H-{\rm tr}_g[\theta,H^{-1}\bar{\theta}^TH])\\
  &=\tilde{\Delta}H+4{\rm tr}_g\bar{\partial}HH^{-1}\partial H+4(\lambda H-{\rm tr}_g[\theta,H^{-1}\bar{\theta}^TH]).
\end{split}
\end{equation}
From (\ref{NAHYMF}), one can see that the flow is a non-linear strictly parabolic equation.

\begin{prop}\label{PROP3_1}
  Let $H(t)$ be a solution of the affine Hermitian-Yang-Mills flow (\ref{AHYMF}), then
  \begin{equation}\label{HKOTMC}
  (\frac{\partial}{\partial t}-\tilde{\Delta})({\rm tr}({\rm tr}_g(F_{H(t)}+[\theta,\theta^{*H(t)}])-\lambda\cdot {\rm Id}_E))=0
  \end{equation}
  and
  \begin{equation}\label{HKOMC}
 (\frac{\partial}{\partial t}-\tilde{\Delta})|{\rm tr}_g(F_{H(t)}+[\theta,\theta^{*{H(t)}}])-\lambda\cdot {\rm Id}_E|_{H(t)}^2\leq 0.
  \end{equation}
\end{prop}
\begin{proof}
  Set $h(t)=H_0^{-1}H(t)$. For simplicity, we denote $\Phi(t)=\Phi(H(t),\theta)={\rm tr}_g(F_{H(t)}+[\theta,\theta^{*{H(t)}}])-\lambda \cdot {\rm Id}_E$ and omit $t$ in the calculations.

  Using the identities
  \begin{equation}
    \partial_H-\partial_{H_0}=h^{-1}\partial_{H_0}h,
  \end{equation}
   \begin{equation}
    F_H-F_{H_0}=\bar{\partial}_E(h^{-1}\partial_{H_0}h),
    \end{equation}
    \begin{equation}
    \theta^{*H}=h^{-1}\theta^{*H_{0}}h,
 \end{equation}
  we have
  \begin{equation}\label{PPTOF}
\begin{split}
    \frac{\partial}{\partial t}\Phi(H,\theta)&=\frac{\partial}{\partial t}({\rm tr}_g(F_H+[\theta,\theta^{*H}])-\lambda\cdot {\rm Id}_E)\\
    &={\rm tr}_g(\bar{\partial}_E(\partial_H(h^{-1}\frac{\partial h}{\partial t}))+[\theta,[\theta^{*H},h^{-1}\frac{\partial h}{\partial t}]])
    \end{split}
\end{equation}
  and
  \begin{equation}\label{LPLOFS}
  \begin{split}
    \tilde{\Delta}|\Phi(H,\theta)|_H^2=&-4{\tr}_g\bar{\partial}\partial{\rm tr}(\Phi H^{-1}\bar{\Phi}^TH)\\
    =&-4{\rm tr}_g\bar{\partial}{\rm tr}(\partial \Phi H^{-1}\bar{\Phi}^TH-\Phi H^{-1}\partial H H^{-1}\bar{\Phi}^TH\\
    &+\Phi H^{-1}\overline{\bar{\partial}\Phi}^TH+\Phi H^{-1}\bar{\Phi}^TH H^{-1}\partial H)\\
    =&2{\rm Re}\langle-4{\rm tr}_g\bar{\partial}_E\partial_H\Phi,\Phi\rangle_H+4|\partial_H\Phi|_H^2+4|\bar{\partial}_E\Phi|_H^2.
  \end{split}
\end{equation}
According to (\ref{PPTOF}), one can easily check that
\begin{equation}\label{PPTOTF}
\frac{\partial}{\partial t}{\rm tr}\Phi={\rm tr}\left(\frac{\partial \Phi}{\partial t}\right)={\rm tr}_g\left({\rm tr}\bar{\partial}\partial (h^{-1}\frac{\partial h}{\partial t}) \right)=\tilde{\Delta}{\rm tr}\Phi.
\end{equation}
This implies (\ref{HKOTMC}).

A simple computation gives us that
\begin{equation}\label{PPTOFS}
\begin{split}
  \frac{\partial}{\partial t}|\Phi(H,\theta)|_H^2&=\frac{\partial }{\partial t}{\rm tr}(\Phi H^{-1}\bar{\Phi}^TH)\\
  &=2{\rm Re}\langle-4{\rm tr}_g\bar{\partial}_E\partial_H\Phi,\Phi\rangle_H-8\langle{\rm tr}_g[\theta,[\theta^{*H},\Phi]],\Phi \rangle_H.
\end{split}
\end{equation}
From (\ref{LPLOFS}) and (\ref{PPTOFS}), we can conclude that
\begin{equation}
\begin{split}
  (\frac{\partial }{\partial t}-\tilde{\Delta})|\Phi(H,\theta)|_H^2&=-8\langle{\rm tr}_g[\theta,[\theta^{*H},\Phi]],\Phi\rangle_H-4|\partial_H\Phi|_H^2-4|\bar{\partial}_E\Phi|_H^2\\
  &\leq 0.
\end{split}
\end{equation}

\end{proof}

Next, we will recall the Donaldson's distance on the space of the Hermitian metrics.
\begin{defn}
  For any two Hermitian metrics $H$ and $K$ on the complex vector bundle $E$, we define
  \begin{equation}
  \sigma(H,K)={\rm tr}(H^{-1}K)+{\rm tr}(K^{-1}H)-2{\rm rank} E.
  \end{equation}
\end{defn}

It is obvious that $\sigma(H,K)\geq 0$ with the equality if and only if $H=K$. A sequence of Hermitian metrics $H_i$ converge to $H$ in the $C^0$-topology if and only if $\sup\limits_{M}\sigma(H_i,H)\rightarrow 0$ as $i\rightarrow \infty$.

Denote $h=K^{-1}H$, then we know
\begin{equation}\label{THMHMMK}
{\rm tr}\left(h({\rm tr}_g F_H-{\rm tr}_g F_K)\right)=-\frac{1}{4}\tilde{\Delta}{\rm tr} h+{\rm tr}(-{\rm tr}_g\bar{\partial}h h^{-1}\partial_Kh).
\end{equation}

Similarly,
\begin{equation}\label{THIMKMMH}
{\rm tr}\left(h^{-1}({\rm tr}_gF_K-{\rm tr}_gF_H)\right)=-\frac{1}{4}\tilde{\Delta}{\rm tr} h^{-1}+{\rm tr}(-{\rm tr}_g\bar{\partial}h^{-1}h\partial_Hh^{-1}).
\end{equation}

%Since $h$ is a positive Hermitian endomorphism, we can follow Donaldson's arguments in \cite{Do85} to check that ${\rm tr}(-{\rm tr}_g\bar{\partial}h h^{-1}\partial_Kh)\geq 0$.
So we also have the following proposition.

\begin{prop}\label{Prop3_4}
  Let $H$ and $K$ be two affine Hermitian-Einstein metrics, then $\sigma(H,K)$ is subharmonic with respect to the affine Laplace operator, i.e.
  \begin{equation}\label{TLPLDD}
  \tilde{\Delta}\sigma(H,K)\geq 0.
  \end{equation}
\end{prop}

Let $H(t)$ and $K(t)$ be two solutions of the affine Hermitian-Yang-Mills flow (\ref{AHYMF}), and denote $h(t)=K^{-1}(t)H(t)$, then using  (\ref{THMHMMK}) and (\ref{THIMKMMH}), we obtain
\begin{equation}\label{TLPLTH}
\tilde{\Delta} {\rm tr} h(t)=4({\rm tr}(-{\rm tr}_g\bar{\partial}h h^{-1}\partial_Kh)-{\rm tr}(h({\rm tr}_gF_H-{\rm tr}_gF_K)))
\end{equation}
and
\begin{equation}
\tilde{\Delta} {\rm tr} h^{-1}(t)=4({\rm tr}(-{\rm tr}_g\bar{\partial}h^{-1}h\partial_Hh^{-1})-{\rm tr}(h^{-1}({\rm tr}_gF_K-{\rm tr}_gF_H))).
\end{equation}
By direct calculations, one can get
\begin{equation}
\begin{split}
  \frac{\partial}{\partial t}{\rm tr} h(t)&={\rm tr}(\frac{\partial h}{\partial t})={\rm tr}(\frac{\partial}{\partial t}(K^{-1}H))\\
  &={\rm tr}(-K^{-1}\frac{\partial K}{\partial t}K^{-1}H+K^{-1}\frac{\partial H}{\partial t})\\
  &=-4{\rm tr}(h({\rm tr}_gF_H-{\rm tr}_gF_K))-4{\rm tr}(h({\rm tr}_g[\theta,\theta^{*H}-\theta^{*K}])).
\end{split}
\end{equation}
Similarly,
\begin{equation}\label{PPTTHI}
\frac{\partial}{\partial t}{\rm tr} h^{-1}=-4{\rm tr}\left(h^{-1}({\rm tr}_gF_K-{\rm tr}_gF_H)\right)-4{\rm tr}(h^{-1}({\rm tr}_g[\theta,\theta^{*K}-\theta^{*H}])).
\end{equation}
From (\ref{TLPLTH})-(\ref{PPTTHI}), it follows that
\begin{equation}
\begin{split}
  &(\tilde{\Delta}-\frac{\partial}{\partial t})({\rm tr} h(t)+{\rm tr} h^{-1}(t))\\
  =&4{\rm tr}(-{\rm tr}_g\bar{\partial}hh^{-1}\partial_Kh)+4{\rm tr}(-{\rm tr}_g\bar{\partial}h^{-1}h\partial_Hh^{-1})\\
  &+4{\rm tr}(h({\rm tr}_g[\theta,\theta^{*H}-\theta^{*K}]))+4{\rm tr}(h^{-1}({\rm tr}_g[\theta,\theta^{*K}-\theta^{*H}]))\\
  \geq &0,
\end{split}
\end{equation}
where we have used
\begin{equation}
{\rm tr}(h({\rm tr}_g[\theta,\theta^{*H}-\theta^{*K}]))=|\theta h^{\frac{1}{2}}-h\theta h^{-\frac{1}{2}}|_K^2
\end{equation}
and
\begin{equation}
{\rm tr}(h^{-1}({\rm tr}_g[\theta,\theta^{*K}-\theta^{*H}]))=|h^{-\frac{1}{2}}\theta-h^{\frac{1}{2}}\theta h^{-1}|_K^2.
\end{equation}
So we have proved the following proposition.
\begin{prop}\label{HKOTAHYMM}
  Let $H(t)$ and $K(t)$ be two solutions of affine Hermitian-Yang-Mills flow (\ref{AHYMF}), then
  \begin{equation}
  (\tilde{\Delta}-\frac{\partial}{\partial t})\sigma(H(t),K(t))\geq 0.
  \end{equation}
\end{prop}

\subsection{The long-time existence of the affine Hermitian-Yang-Mills flow }

Let $(M,D,g)$ be a compact affine Gauduchon manifold (with possibly non-empty boundary), and $(E,\nabla,\theta)$ be a flat Higgs bundle with the initial Hermitian metric $H_0$. If $M$ is closed, then we consider the following evolution equation:
\begin{equation}\label{IPOAHYMF}
\left\{\begin{split}
  &H^{-1}(t)\frac{\partial H(t)}{\partial t}=-4({\rm tr}_g(F_{H(t)}+[\theta,\theta^{*{H(t)}}])-\lambda \cdot {\rm Id}_E),\\
  &H(0)=H_0.
  \end{split}\right.
\end{equation}
If $M$ is a compact manifold with non-empty smooth boundary $\partial M $, and the affine Gauduchon metric $g$ is smooth and non-degenerate on the boundary, for given data $\tilde{H}$ on $\partial M$, we consider the following Dirichlet boundary value problem:
\begin{equation}\label{DPOAHYMF}
\left\{\begin{split}
  &H(t)^{-1}\frac{\partial H(t)}{\partial t}=-4({\rm tr}_g(F_{H(t)}+[\theta,\theta^{*{H(t)}}])-\lambda\cdot {\rm Id}_E),\\
  &H(0)=H_0,\\
  &H|_{\partial M}=\tilde{H},
  \end{split}\right.
\end{equation}
where $H_0$ satisfies the boundary condition. Based on the formula (\ref{NAHYMF}), we know the above equations are non-linear strictly parabolic equation, so the standard parabolic theory gives the short-time existence:

\begin{prop}\label{PropSTS}
  For sufficiently small $T>0$, the equations (\ref{IPOAHYMF}) and (\ref{DPOAHYMF}) have a smooth solution defined for $0<t<T$.
\end{prop}

Next, according to the arguments of \cite{Do85}, \cite{Sim88} and \cite{ZX05}, we can show the long-time existence.

\begin{thm}\label{Thm3_6} {\rm (}\protect {\cite[Theorem 3.2]{ZX05}}{\rm )}
  Suppose that a smooth solution $H(t)$ of the evolution equation (\ref{IPOAHYMF}) or (\ref{DPOAHYMF}) is defined for $0\leq t <T<+\infty$. Then $H(t)$ converge in $C^0$-topology to some continuous non-degenerate metric as $t\rightarrow T$.
\end{thm}

\begin{proof}
  Given $\epsilon >0$, by the continuity at $t=0$ we can find a $\delta$ such that
  \begin{equation*}
  \sup\limits_M\sigma(H(t_0),H(t_0'))<\epsilon
  \end{equation*}
  for $0<t_0,t_0'<\delta$. Then Proposition \ref{HKOTAHYMM} and the maximum principle imply that
  \begin{equation*}
    \sup\limits_{M}\sigma(H(t),H(t'))<\epsilon
  \end{equation*}
  for all $t,t'> T-\delta$. This means that $H(t)$ are uniform Cauchy sequence and converge to a continuous limiting metric $H_T$. On the other hand, by Proposition \ref{PROP3_1}, we know that
  \begin{equation}
  \sup\limits_{M\times[0,T)}|{\rm tr}_g(F_{H(t)}+[\theta,\theta^{*{H(t)}}])-\lambda \cdot {\rm Id}_E|<C,
  \end{equation}
  where $C$ is a uniform constant only depending on the initial data $H_0$. A direct calculation yields
  \begin{equation}
 |\frac{\partial}{\partial t}(\log {\rm tr} h(t))|\leq 4|{\rm tr}_g(F_{H(t)}+[\theta,\theta^{*{H(t)}}])-\lambda\cdot {\rm Id}_E|_{H(t)}
  \end{equation}
  and
  \begin{equation}
    |\frac{\partial }{\partial t}(\log {\rm tr} h^{-1}(t))|\leq 4|{\rm tr}_g(F_{H(t)}+[\theta,\theta^{*{H(t)}}])-\lambda\cdot {\rm Id}_E|_{H(t)}.
  \end{equation}
  Using above formulas, one can conclude that $\sigma(H(t),H_0)$ are bounded uniformly on $M\times[0,T)$, therefore $H_T$ is a non-degenerate metric.
\end{proof}

Then we can also obtain the following results in the affine Gauduchon case.

\begin{lem}\label{Lem3_7}
   {\rm (}\protect {\cite[Lemma 3.3]{ZX05}}{\rm )}Let $(M,D,g)$ be a compact affine Gauduchon manifold without boundary $($with non-empty boundary$)$, $(E,\nabla,\theta)$ be the flat Higgs bundle with the initial Hermitian metric $H_0$. Let $H(t)$, $0\leq t<T$, be any one-parameter family of Hermitian metrics on $(E,\nabla,\theta)$ $($and satisfy the Dirichlet boundary condition$)$. If $H(t)$ converge in $C^0$-topology to some continuous metric $H_T$ as $t\rightarrow T$, and if $\sup\limits_{M}|{\rm tr}_gF_{H(t),\theta}|_{H_0}$ is bounded uniformly in $t$, then $H(t)$ are bounded in $C^1$ and also bounded in $L_2^p$ $($ for any $1<p<+\infty)$  uniformly in t.
\end{lem}

\begin{thm}\label{LTSOC}{\rm (}\protect {\cite[Theorem 3.4]{ZX05}}{\rm )}
  The evolution equations (\ref{IPOAHYMF}) and (\ref{DPOAHYMF}) have a unique solution $H(t)$ which exists for $0\leq t<+\infty$.
\end{thm}
%\begin{proof}
%  Proposition \ref{PropSTS} guarantees that a solution exists for a short time. Suppose that the solution $H(t)$ exists for $0\leq t<T<+\infty$. By Theorem \ref{Thm3_6}, $H(t)$ converge in $C^0$-topology to a non-degenerate continuous limit metric $H(T)$ as $t\rightarrow T$. Since $t<+\infty$, from (\ref{HKOMC}) and the maximum principle, we conclude that $|{\rm tr}_gF_{H,\theta}-\lambda\cdot {\rm Id}_E|_H$ are bounded independent of $t$. Moreover, $|{\rm tr}_gF_{H,\theta}|^2_{H_{0}} $ are bounded independent of $t$. Hence by Lemma \ref{Lem3_7}, $H(t)$ are bounded in $C^1$  and also bounded in $L_2^p$ $($ for any $1<p<+\infty)$ uniformly in $t$. Since the evolution equations (\ref{IPOAHYMF}) and (\ref{DPOAHYMF}) are quadratic in the first derivative of $H$, we can apply Hamilton's method \cite{Ha75} to deduce that $H(t)\rightarrow H(T)$ in $C^{\infty}$, and the solution can be continued past $T$. Then the evolution equations (\ref{IPOAHYMF}) and (\ref{DPOAHYMF}) have a solution $H(t)$ defined for all time.
%
%  By Proposition \ref{HKOTAHYMM} and maximum principle, it is easy to conclude the uniqueness of the solution.
%\end{proof}

\subsection{The Dirichlet boundary problem for affine Hermitian-Einstein metric}
Since we have proved the long-time existence of (\ref{DPOAHYMF}), it remains for us to show that $H(t)$ will converge to the affine Hermitian-Einstein metric which we want. In this section, we will discuss the Dirichlet boundary problem for affine Hermitian-Einstein metric by using the heat equation method to deform an arbitrary initial metric to the desired one. Similar to the discussions of Donaldson \cite{Do92} and Simpson \cite{Sim88}, we also proved the following theorem.
\begin{thm}
  Let $(E,\nabla, \theta)$ be a flat Higgs bundle over the compact affine manifold $\overline{M}$ with non-empty boundary $\partial M$. Then for any Hermitian metric $\tilde{H}$ on the restriction of $E$ to $\partial M$, there
is a unique  affine Hermitian-Einstein metric $H$ on $E$ such that $H=\tilde{H}$ on $\partial M$.
\end{thm}
\begin{proof}
Suppose $H(t)$ is the solution of (\ref{DPOAHYMF}) for $0\leq t<\infty$. By direct calculations, one can easily check that
\begin{equation}\label{Deta}
  |d|\eta|_H|^2\leq |\nabla_H\eta|_H^2
\end{equation}
for any section $\eta \in \Gamma({\rm End}(E))$.
This together with the formula (\ref{HKOMC}) in Proposition \ref{PROP3_1} gives out
\begin{equation}\label{DMPPT1}
  (\tilde{\Delta}-\frac{\partial}{\partial t})|{\rm tr}_g(F_{H(t)}+[\theta,\theta^{*H(t)}])-\lambda{\rm Id}_E|_{H(t)}\geq 0.
\end{equation}
According to \protect{\cite[Proposition 1.8]{Taylor}}, we can solve the following Dirichlet problem on $\overline{M}$:
\begin{equation}\label{DPOV}
\left\{\begin{split}
    &\tilde{\Delta}v=-|{\rm tr}_g(F_{H_0}+[\theta,\theta^{*H_0}])-\lambda{\rm Id}_E|_{H_0},\\
    &v|_{\partial M}=0.
  \end{split}\right.
\end{equation}
Then set $w(x,t)=\int_0^t|{\rm tr}_g(F_{H}+[\theta,\theta^{*H}])-\lambda{\rm Id}_E|_H(x,s)ds-v(x)$, where $v$ is the solution of the Dirichlet problem (\ref{DPOV}). From (\ref{DMPPT1}), (\ref{DPOV}) and the boundary condition satisfied by $H(t)$, we see that $|{\rm tr}_g(F_{H}+[\theta,\theta^*])-\lambda{\rm Id}_E|_{H}(x,t)$ vanishes on $\partial M$ for $t>0$. Then clearly $w(x,t)$ satisfies
\begin{equation}\label{ThreeCondition}
\left\{\begin{split}
    &(\tilde{\Delta}-\frac{\partial}{\partial t})w(x,t)\geq 0,\\
    &w(x,0)=-v(x),\\
    &w(x,t)|_{\partial M}=0.
  \end{split}\right.
\end{equation}
Applying the maximum principle, we have
\begin{equation}\label{I0tt}
  \int_0^t|{\rm tr}_g(F_{H}+[\theta,\theta^{*H}])-\lambda{\rm Id}_E|_{H}(x,s)ds \leq \sup\limits_{y\in \overline{M}} v(y),
\end{equation}
for any $x \in \overline{M}$ and $0< t \leq \infty$.

Let $0\leq t_1\leq t$, $\hat{h}(t)=H^{-1}(t_1)H(t)$. It is obvious that
\begin{equation}
  \hat{h}^{-1}(t)\frac{\partial \hat{h}(t)}{\partial t}=-4({\rm tr}_g(F_{H(t)}+[\theta,\theta^{*H(t)}])-\lambda{\rm Id}_E).
\end{equation}
Then
\begin{equation}\label{PPTOLTHH}
  \frac{\partial}{\partial t}\log{\rm tr}(\hat{h}(t)) \leq 4|{\rm tr}_g(F_{H(t)}+[\theta,\theta^{*H(t)}])-\lambda{\rm Id}_E|_{H(t)}.
\end{equation}
Integrating this from $t_1$ to $t$ gives us
\begin{equation}
  {\rm tr}(H^{-1}(t_1)H(t))\leq r\exp(4\int_{t_1}^t|{\rm tr}_g(F_{H}+[\theta,\theta^{*H}])-\lambda{\rm Id}_E|_{H}(s)ds).
\end{equation}
Similarly, we can get the estimate for ${\rm tr}(H^{-1}(t)H(t_1))$. Combining them together, we can deduce that
\begin{equation}\label{UODF}
  \sigma(H(t),H(t_1))\leq 2r(\exp(4\int_{t_1}^t|{\rm tr}_g(F_{H}+[\theta,\theta^{*H}])-\lambda {\rm Id}_E|_{H}(s)ds)-1).
\end{equation}
 Because of (\ref{I0tt}), we conclude that $H(t)$  converge in $C^0$-topology to some continuous metric $H_{\infty}$ as $t\rightarrow \infty$. Using Lemma \ref{Lem3_7}, we know that $H(t)$ have uniform $C^1$ and $L_2^p(1<p<\infty)$ bounds. On the other hand, $|H^{-1}(t)\frac{\partial H(t)}{\partial t}|$ is bounded uniformly. This together with the standard elliptic regularity shows that there exists a subsequence $H(t)\rightarrow H_{\infty}$ in $C^{\infty}$-topology as $t\rightarrow \infty$. Of course (\ref{I0tt}) means that
\begin{equation}
  {\rm tr}_g(F_{H_{\infty}}+[\theta,\theta^{*H\infty}])=\lambda {\rm Id}_E,
\end{equation}
i.e. $H_{\infty}$ is the desired affine Hermitian-Einstein metric satisfying the boundary condition. From Proposition \ref{Prop3_4} and the maximum principle, the uniqueness of the solution follows.
\end{proof}

\section{The affine Hermitian-Yang-Mills flow on non-compact affine Gauduchon manifold}

In this section, we will discuss the affine Hermitian-Yang-Mills flow on some non-compact manifold. Before the discussion, we need some important propositions. First, we derive the local $C^1$-estimate. Then, we want to get following identity in the non-compact affine Gauduchon case:
\begin{equation}\label{KI}
\int_{M}{\rm tr}\left(\Phi(H_0,\theta)s\right)\frac{\omega^n_g}{\nu}+\int_{M}\langle\Psi(s)(D''s),D''s\rangle_{H_0}\frac{\omega^n_g}{\nu}=\int_{M}{\rm tr}\left(\Phi(H,\theta)s\right)\frac
  {\omega^n_g}{\nu}.
\end{equation}

Consider the following affine Hermitian-Yang-Mills flow
\begin{equation}\label{AHYMF2}
 H^{-1}(t)\frac{\partial H(t)}{\partial t}=-4({\rm tr}_g(F_{H(t)}+[\theta,\theta^{*H(t)}])-\lambda\cdot {\rm Id}_E).
\end{equation}

\begin{prop}\label{LC_1}
  Suppose $H(t)$ is a long-time solution of the flow (\ref{AHYMF2}) on the compact affine manifold ${\overline{M}}$(with nonempty smooth boundary $\partial M$). Set $h(t)=H_0^{-1}H(t)$ and assume there exists a constant $\bar{C}_0$ such that
  \begin{equation}\label{C0Condition}
  \sup\limits_{(x,t)\in \overline{M}\times[0,+\infty)}|\log h(t)|_{H_0}\leq {\bar{C}_0}.
  \end{equation}
  Then, for any compact subset $\Omega \subset \overline{M}$, there exists a uniform constant $\bar{C}_1$ depending only on $\bar{C}_0,d^{-1}$ and  the geometry of $\tilde{\Omega}$ such that
  \begin{equation}\label{C1result}
   \sup\limits_{(x,t)\in \Omega\times[0,+\infty)}|h^{-1}(t)\partial_{H_0}h(t)|_{H_0}\leq {\bar{C}_1},
  \end{equation}
  where $d$ is the distance of $\Omega$ to $\partial M$, and $\tilde{\Omega}=\{x\in \overline{M}|{\rm dist}(x,\Omega)\leq \frac{1}{2}d\}$.
\end{prop}

\begin{proof}
  We will follow the argument in \protect {\cite[Lemma 2.4]{LZZ17}} to get the local uniform $C^1$-estimate. For simplicity, we denote $\mathcal{T}(t)=h^{-1}(t)\partial_{H_0}h(t)$ and $\Phi(H(t),\theta)={\rm tr}_gF_{H(t),\theta}-\lambda\cdot {\rm Id}_E$. Then the direct computations give us that
  \begin{equation}\label{HKOTH}
  \begin{split}
    (\tilde{\Delta}-\frac{\partial}{\partial t}){\rm tr}h&=4{\rm tr}(-{\rm tr}_g\bar{\partial}hh^{-1}\partial_{H_0}h)+4{\rm tr}(h\Phi(H_0,\theta))
    +4{\rm tr}(h[\theta,\theta^{H}-\theta^{*H_0}])\\
    &=4{\rm tr}(-{\rm tr}_g\bar{\partial}hh^{-1}\partial_{H_0}h)+4{\rm tr}(h\Phi(H_0,\theta))
    +4{\rm tr}([\theta,h]\wedge h^{-1}[h,\theta^{*H_0}])\\
    &\geq 4{\rm tr}(-{\rm tr}_g\bar{\partial}hh^{-1}\partial_{H_0}h)+4{\rm tr}(h\Phi(H_0,\theta)),
  \end{split}
  \end{equation}
  \begin{equation}
    \frac{\partial}{\partial t}\mathcal{T}=\frac{\partial}{\partial t}(H^{-1}\partial H)=\partial_{H_0}(h^{-1}\frac{\partial h}{\partial t})=-4\partial_{H}\left(\Phi(H,\theta)\right),
  \end{equation}
 and
\begin{equation}\label{HKOTS}
\begin{split}
      (\tilde{\Delta}-\frac{\partial}{\partial t})|\mathcal{T}|_H^2\geq & |\nabla_H\mathcal{T}|_H^2-\hat{C}_1(|{\rm tr}_gF_{H_0}|_H+|F_{H_0}|_H+|\theta|_H^2+|Ric(g)|_g+|\nabla_gJ|)|\mathcal{T}|^2_H\\
      &-\hat{C}_2|\nabla_{H_0}{\rm tr}_gF_{H_0}|_H|\mathcal{T}|_H-4|\nabla_{H_0}\theta|_H^2,
    \end{split}
    \end{equation}
  where $J$ is the complex structure on $M^{\mathbb{C}}$ and  positive constants $\hat{C}_1,\hat{C}_2$ depend only on the dimension $n$ and rank $r$. By (\ref{HKOTS}) and Proposition \ref{PROP3_1}, we have
  \begin{equation}\label{NHKOTS}
  (\tilde{\Delta}-\frac{\partial}{\partial t})|\mathcal{T}|_H^2\geq |\nabla_H\mathcal{T}|_H^2-\hat{C}_3|\mathcal{T}|_H^2-\hat{C}_3
  \end{equation}
  on the domain $\tilde{\Omega}\times[0,+\infty)$, where $\hat{C}_3$ is a uniform constant depending only on $\bar{C}_0$, $\sup_{\tilde{\Omega}}|{\rm tr}_gF_{H_0}|_{H_0}$, $\sup_{\tilde{\Omega}}|F_{H_0}|_{H_0}$, $\sup_{\tilde{\Omega}}|\nabla_{H_0}{\rm tr}_gF_{H_0}|_{H_0}$, $\sup_{\tilde{\Omega}}|\nabla_{H_0}\theta|_{H_0}$, $\sup_{\tilde{\Omega}}|\theta|_{H_0}$ and the geometry of $\tilde{\Omega}$.

 Set $\bar{\Omega}=\left\{x\in \overline{M}|\,{\rm dist}(x,\Omega) \leq \frac{1}{4}d\right\}$. Let $\varphi_1$, $\varphi_2$ be  the non-negative cut-off functions satisfying:
  \begin{equation*}
  \varphi_1=\left\{\begin{split}
    1,\ \ &x\in \Omega,\\
    0,\ \ &x\in M\setminus\bar{\Omega},
  \end{split}\right.
  \end{equation*}
  \begin{equation*}
     \varphi_2=\left\{\begin{split}
    1,\ \ &x\in\bar{\Omega},\\
    0,\ \ &x\in M\setminus \tilde{\Omega},
  \end{split}\right.
  \end{equation*}
 and
 \begin{equation*}
   |d\varphi_i|^2+|\tilde{\Delta}\varphi_i|\leq c,i=1,2,
 \end{equation*}
where $c$ is a constant depending only on $d^{-2}$. Consider the following test function
\begin{equation}
  f(\cdot,t)=\varphi_1^2|\mathcal{T}|_{H}^2+W\varphi_2^2{\rm tr}h,
\end{equation}
where the constant $W$ will be chosen large enough later. It follows from (\ref{HKOTH}) and (\ref{HKOTS}) that
\begin{equation}
\begin{split}
  (\tilde{\Delta}-\frac{\partial}{\partial t})f\geq & \varphi_1^2(|\nabla_H\mathcal{T}|_H^2-\hat{C}_3|\mathcal{T}|_H^2-\hat{C}_3)+\tilde{\Delta}
  \varphi_1^2|\mathcal{T}|_H^2\\
  &+4\langle\varphi_1\nabla \varphi_1,\nabla|\mathcal{T}|_H^2\rangle\\
  &+W\tilde{\Delta}\varphi_2^2{\rm tr}h+4W\langle\varphi_2\nabla\varphi_2,\nabla{\rm tr}h\rangle\\
  &+4W\varphi_2^2{\rm tr}({\rm tr}_gh^{-1}\partial_{H_0}h\bar{\partial}h+h\Phi(H_0,\theta)).
\end{split}
\end{equation}
Noting that
\begin{equation}
\begin{split}
  4\langle\varphi_1\nabla\varphi_1,\nabla|\mathcal{T}|_H^2\rangle&\geq -8\varphi_1|\nabla \varphi_1||\mathcal{T}|_H|\nabla_H\mathcal{T}|_H\\
  &\geq-\varphi_1^2|\nabla_H\mathcal{T}|_H^2-16|\nabla\varphi_1|^2|\mathcal{T}|_H^2,
\end{split}
\end{equation}
\begin{equation}
  2W\langle\varphi_2\nabla\varphi_2,\nabla{\rm tr}h\rangle \geq-\varphi_2^2|\nabla{\rm tr}h|^2_H-W^2|\nabla\varphi_2|^2,
\end{equation}
and
\begin{equation}
\begin{split}
  |\mathcal{T}|^2_H&={\rm tr}({\rm tr}_gh^{-1}\partial_{H_0}hH^{-1}\overline{(h^{-1}\partial_{H_0}h)}^TH)\\
  &={\rm tr}({\rm tr}_gh^{-1}\partial_{H_0}h h^{-1}\bar{\partial}h)\\
  &\leq e^{\bar{C}_0}{\rm tr}({\rm tr}_gh^{-1}\partial_{H_0}h\bar{\partial}h).
\end{split}
\end{equation}
Then there holds that
\begin{equation}
  (\tilde{\Delta}-\frac{\partial}{\partial t})f\geq\varphi_2^2(-\hat{C}_3-18c-4e^{2\bar{C}_0}+4We^{-\bar{C}_0})|\mathcal{T}|_H^2-\tilde{C}_0,
\end{equation}
where $\tilde{C}_0$ is a positive constant depending only on $\bar{C}_0$ and $\hat{C}_3$. If we choose
\begin{equation}
  W=\frac{1}{4}e^{-\bar{C}_0}(\hat{C}_3+18c+4e^{2\bar{C}_0}+1),
\end{equation}
then we can obtain
\begin{equation}\label{HKOf}
(\tilde{\Delta}-\frac{\partial }{\partial t})f\geq \varphi_2^2|\mathcal{T}|^2_H-\tilde{C}_0
\end{equation}
on $M\times[0,+\infty)$. Let $f(q,t_0)=\max\limits_{M\times[0,+\infty)}f$. On the basis of the definition of $\varphi_i$ and the local uniform $C^0$-bound of $h(t)$, we may assume that
\begin{equation*}
  (q,t_0)\in \overline{M}\times(0,+\infty).
\end{equation*}
Of course the inequality (\ref{HKOf}) yields
\begin{equation}
|\mathcal{T}(t_0)|^2_{H(t_0)}(q)\leq \tilde{C}_0.
\end{equation}
So there exists a uniform constant $\bar{C}_1$ depending only on $\bar{C}_0,d^{-1}$ and the geometry of $\tilde{\Omega}$ such that
\begin{equation*}
  \sup\limits_{(x,t)\in \Omega\times[0,+\infty)}|h^{-1}(t)\partial_{H_0}h(t)|_{H_0}\leq {\bar{C}_1}.
\end{equation*}
\end{proof}

Now we want to prove the key identity (\ref{KI}), which plays an important role in the proof of Theorem \ref{Maintheorem}.

Suppose $(M,D,g,\nu)$ is a compact affine Gauduchon manifold with non-empty smooth boundary $\partial M$. Let $\phi$ be  a smooth function defined on $M$ and satisfy the boundary condition $\phi|_{\partial M}=a$, where $a$ is a constant. Integrating by parts and applying Stokes' formula, one can obtain
\begin{equation}\label{bformula}
\begin{split}
  &\int_M|d\phi|^2\frac{\omega^n_g}{\nu}\\
  =&2n\int_M\frac{\partial \phi \wedge \bar{\partial}\phi\wedge \omega^{n-1}_g}{\nu}\\
  =&2n\int_M\frac{\partial(\phi\bar{\partial}\phi)\wedge \omega^{n-1}_g}{\nu}-2n\int_M\phi\frac{\partial\bar{\partial}\phi\wedge \omega^{n-1}_g}{\nu}\\
  =&2n\int_{M}(a-\phi)\frac{\partial\bar{\partial}\phi\wedge \omega^{n-1}_g}{\nu}-2n\int_M\frac{\partial((a-\phi)\bar{\partial}\phi)\wedge\omega^{n-1}_g}{\nu}\\
  =&\frac{1}{2}\int_M(a-\phi)\tilde{\Delta}\phi\frac{\omega^n_g}{\nu}+2n\int_M\frac{\partial(\bar{\partial}(a-\phi)^2\wedge\omega^{n-1}_g)}{\nu}+2n\int_M
  \frac{\bar{\partial}(a-\phi)^2\wedge\partial\omega^{n-1}_g}{\nu}\\
  =&\frac{1}{2}\int_M(a-\phi)\tilde{\Delta}\phi\frac{\omega^n_g}{\nu}+n\int_Md(\frac{\bar{\partial}(a-\phi)^2\wedge\omega^{n-1}_g}{\nu})+n\int_Md
  (\frac{(a-\phi)^2\wedge \partial \omega^{n-1}_g}{\nu})\\
  =&\frac{1}{2}\int_M(a-\phi)\tilde{\Delta}\phi\frac{\omega^n_g}{\nu}.
\end{split}
\end{equation}
Then similar to the argument in \cite{Sim88}, we can get the following lemma.

\begin{lem}\label{SFLem}
  Assume $(X,D,g,\nu)$ is a special non-compact affine Gauduchon manifold admitting an exhaustion function $\phi$ with $\int_X|\tilde{\Delta}\phi|\frac{\omega^{n}_g}{\nu}<+\infty$, and suppose $\eta$ is an $(n-1)$-form with $\int_X|\eta|^2\frac{\omega^n_g}{\nu}<+\infty$. Then if $d\eta$ is integrable, we have
  \begin{equation*}
      \int_Xd\eta=0.
  \end{equation*}
\end{lem}

%\begin{proof}
%  Take $\phi \geq 0$. Let $X_a$ be the set $X_a=\{x\in X|\phi(x)\leq a\}$ and $Y_a$ be the boundary $\{y\in X|\phi(y)=a\}$. Clearly (\ref{bformula})  means that
%  \begin{equation*}
%    \int_{X_a}|d\phi|^2\frac{\omega^n_g}{\nu}=\frac{1}{2}\int_{X_a}(a-\phi)\tilde{\Delta}\phi\frac{\omega^n_g}{\nu}\leq Ca.
%  \end{equation*}
%On the other hand,
%\begin{equation*}
%  \frac{d}{da}\int_{X_a}|d\phi|^2=\int_{Y_a}|d\phi|.
%\end{equation*}
%So
%\begin{align}
%  \int_0^A\int_{Y_a}|d\phi|dt&=\int_0^A\frac{d}{da}\int_{X_a}|d\phi|^2da\notag\\
%  &=\int_{X_A}|d\phi|^2-\int_{X_0}|d\phi|^2\notag\\
%  &=\int_{X_A}|d\phi|^2\leq CA.\notag
%\end{align}
%Note that Stokes' formula for the form $\eta$ is
%\begin{equation*}
%  \int_{X_a}d\eta=\int_{Y_a}\eta,
%\end{equation*}
%and the $L^2$ condition is
%\begin{equation*}
%  \int_0^A\int_{Y_a}\frac{|\eta|^2}{|d\phi|}da=\int_{X_A}|\eta|^2 \leq C.
%\end{equation*}
%From H\"older inequality, it immediately can be seen that
%\begin{align}
%  \int_0^A\left|\int_{X_a}d\eta\right| &\leq \int_0^A\int_{Y_a}|\eta|da\notag\\
%  &\leq\left(\int_{X_A}|\eta|^2\right)^{\frac{1}{2}}\left(|d\phi|^2\right)^{\frac{1}{2}}\notag\\
%  &=\left(\int_0^A\int_{Y_a}\frac{|\eta|^2}{|d\phi|}\right)^{\frac{1}{2}}\left(\int_{0}^A\int_{Y_a}|d\phi|\right)^{\frac{1}{2}}\notag\\
%  &\leq CA^{\frac{1}{2}}.\notag
%\end{align}
%Therefore there is a subsequence $a_{i}\rightarrow +\infty$, such that
%\begin{equation*}
%  \int_{X_{a_i}}d\eta\rightarrow 0.
%\end{equation*}
%So if $d\eta $ is integrable, then $\int_{X}d\eta=0$.
%\end{proof}

\begin{prop}
  Let $(E,\nabla,\theta)$ be a flat Higgs bundle with a fixed Hermitian metric $H_0$ over a special affine Gauduchon manifold $(M,g,D,\nu)$. Let $H$ be a Hermitian metric on $E$ and $s:=\log(H_{0}^{-1}H)$. If one of the following three conditions is satisfied:
  \begin{enumerate}
    \item[(1)]Suppose that $M$ is a compact manifold without boundary.
    \item[(2)]Suppose that $M$ is a compact manifold with non-empty smooth boundary $\partial M$, and $H|_{\partial M}=H_{0}|_{\partial M}$.
    \item[(3)]Suppose that $M$ is a non-compact manifold admitting an exhaustion function $\phi$ with  $\int_{M}|\tilde{\Delta}\phi|\frac{\omega^n_g}{\nu}<+\infty$. Furthermore, we also assume that $|\frac{\partial \omega^{n-1}_g}{\nu}|_g\in L^2(M)$,$s\in L^{\infty}(M)$, and $D_{H_0,\theta}^{1,0}s\in L^2(M)$, where $D_{H_0,\theta}^{1,0}=\partial_{H_0}+\theta^{*{H_0}}$.
  \end{enumerate}
  Then, we have the following identity:
  \begin{equation}\label{kiiProp}
  \int_{M}{\rm tr}\left(\Phi(H_0,\theta)s\right)\frac{\omega^n_g}{\nu}+\int_{M}\langle\Psi(s)(D''s),D''s\rangle_{H_0}\frac{\omega^n_g}{\nu}=\int_{M}{\rm tr}\left(\Phi(H,\theta)s\right)\frac
  {\omega^n_g}{\nu},
  \end{equation}
where $D''=\bar{\partial}_E+\theta$ and $\Psi$ is the function defined by
\begin{equation*}
\Psi(x,y)=\left\{\begin{split}
  &\frac{e^{y-x}-1}{y-x},\ \ &x\neq y,\\
  &1,\ \ &x=y.
\end{split}\right.
\end{equation*}
\end{prop}

\begin{proof}
  Set $h=H_{0}^{-1}H=e^s$. Obviously the definition tells us that
  \begin{equation}
      {\rm tr}((\Phi(H,\theta)-\Phi(H_0,\theta))s)=\langle{\rm tr}_g((h^{-1}\partial_{H_0}h)+[\theta,h^{-1}[\theta^{*H_0},h]]),s\rangle_{H_0}.
  \end{equation}
Based on ${\rm tr}(h^{-1}(\partial_{H_0}hs))={\rm tr}(s\partial_{H_0}s)$, ${\rm tr}(s[\theta^{*H_0},s])=0$ and $\partial \bar{\partial}\omega^{n-1}_g=0$, we derive
  \begin{equation}\label{BF}
  \begin{split}
    &\quad \int_M\langle{\rm tr}_g\big(\bar{\partial}(h^{-1}\partial_{H_0}h)\big),s\rangle_{H_0}\frac{\omega^n_g}{\nu}\\
    &=n\int_M\frac{{\rm tr}(\bar{\partial}(h^{-1}\partial_{H_0}h)s)\wedge\omega^{n-1}_g}{\nu}\\
    &=n\int_M\frac{\bar{\partial}{\rm tr}\left(h^{-1}(\partial_{H_0}h)s\right)\wedge\omega^{n-1}_g}{\nu}+n\int_M\frac{{\rm tr}(h^{-1}\partial_{H_0}h\wedge\bar{\partial}s)\wedge \omega^{n-1}_g}{\nu}\\
    &=n\int_M\frac{{\rm tr}(s\partial_{H_0}s)\wedge \bar{\partial}\omega^{n-1}_g}{\nu}+n\int_M\frac{\bar{\partial}\left({\rm tr}(s\partial_{H_0}s)\wedge\omega^{n-1}_g\right)}{\nu}\\
    &\quad+n\int_M\frac{{\rm tr}(h^{-1}\partial_{H_0}h\wedge\bar{\partial}s)\wedge \omega^{n-1}_g}{\nu}\\
    &=n\int_M\frac{\partial\left(\frac{1}{2}{\rm tr}(s^2)\wedge \bar{\partial}\omega^{n-1}_g\right)}{\nu}+n\int_M\frac{\bar{\partial}\left({\rm tr}(sD_{H_0,\theta}^{1,0}s)\wedge \omega^{n-1}_g\right)}{\nu}\\
    &\quad+n\int_M\frac{{\rm tr}(h^{-1}\partial_{H_0}h\wedge \bar{\partial}s)\wedge \omega^{n-1}_g}{\nu}\\
    &=n\int_Md\left(\frac{\frac{1}{2}{\rm tr}(s^2)\wedge{\bar{\partial}}\omega^{n-1}_g}{2\nu}\right)+(-1)^{n-1}\int_Md\left(\frac{{\rm tr}(sD_{H_0,\theta}^{1,0}s)\wedge \omega^{n-1}_g}{2\nu}\right)\\
    &\quad + n\int_M\frac{{\rm tr}(h^{-1}\partial_{H_0}h\wedge\bar{\partial}s)\wedge \omega^{n-1}_g}{\nu}.
  \end{split}
  \end{equation}
  In condition (1), by using Stokes' formula; in condition (2), by using $s|_{\partial M}=0$ and Stokes' formula; in condition (3), by using lemma \ref{SFLem}, we have
  \begin{equation}\label{IMTgB}
    \int_M\langle{\rm tr}_g\big(\bar{\partial}(h^{-1}\partial_{H_0}h)\big),s\rangle_{H_0}\frac{\omega^n_g}{\nu}=\int_M\frac{{\rm tr}(h^{-1}\partial_{H_0}h\wedge \bar{\partial}s)\wedge \omega^{n-1}_g}{\nu}.
  \end{equation}
 According to\protect{ \cite[p.635]{NieZhang}}, we can also prove
 \begin{equation}\label{TTghI}
   {\rm tr}\left({\rm tr}_g(h^{-1}D_{H_{0},\theta}^{1,0}hD''s)\right)=\langle\Psi(s)(D''s),D''s\rangle_{H_0}
 \end{equation}
  and
  \begin{equation}\label{IMTTgT}
    \int_{M}{\rm tr}\left({\rm tr}_g[\theta,h^{-1}[\theta^{*H_0},h]]s\right)\frac{\omega^n_g}{\nu}=\int_M\frac{{\rm tr}\left(h^{-1}[\theta^{*H_0},h][\theta,s]\right)\wedge\omega^{n-1}_g}{\nu}.
  \end{equation}
 From above, there holds that
  \begin{equation}
    \int_M\langle{\rm tr}_g\left(\bar{\partial}(h^{-1}\partial_{H_0}h)+[\theta,h^{-1}[\theta^{*H_0},h]]\right),s\rangle_{H_0}\frac{\omega^n_g}{\nu}
  =\int_M\langle\Psi(s)(D''s),D''s\rangle_{H_0}\frac{\omega^n_g}{\nu}.
  \end{equation}
  So (\ref{kiiProp}) immediately follows.
\end{proof}

\vspace{5mm}
Next, we will consider the long-time existence of the affine Hermitian-Yang-Mills flow (\ref{AHYMF2}) on some non-compact affine Gauduchon manifold $(X,D,g)$ under some conditions of the initial metric $H_0$. In the following, we suppose that $(X,D,g)$ satisfies the Assumption $2$.

  Let $\{X_{\varphi}\}_{\varphi=1}^{\infty}$ be an exhaustion sequence of compact sub-domains of $X$, i.e., they satisfy
  $X_{\varphi}\subset X_{\varphi+1}$ and $\bigcup\limits_{\varphi=1}^{\infty}X_{\varphi}=X$. We consider the  Dirichlet boundary condition:
  \begin{equation}\label{Dirbc}
  H|_{\partial X_{\varphi}}=H_0|_{\partial X_{\varphi}}.
  \end{equation}
  By Theorem \ref{LTSOC}, on every $X_{\varphi}$, there exists a unique long-time solution $H_{\varphi}(t)$ of the following Dirichlet boundary problem:
  \begin{equation}\label{XFAHYMF}
  \left\{\begin{split}
       &H_{\varphi}^{-1}(t)\frac{\partial H_{\varphi}(t)}{\partial t}=-4\left({\rm tr}_g(F_{H_{\varphi}(t)}+[\theta,\theta^{*H_{\varphi}(t)}])-\lambda\cdot {\rm Id}_{E}\right),\\
       &H_{\varphi}(0)=H_0,\\
       &H_{\varphi}(t)|_{\partial X_{\varphi}}=H_{0}|_{\partial X_{\varphi}}.
    \end{split}\right.
  \end{equation}
  For simplicity, set $\Phi(H_{\varphi}(t),\theta)={\rm tr}_g(F_{H_{\varphi}(t)}+[\theta,\theta^{*H_{\varphi}(t)}])-\lambda\cdot {\rm Id}_E$. Suppose that there exists a positive number $C_0$ such that $|\Phi(H_0,\theta)|_{H_0}\leq C_0$ on $X$. Denote $h_{\varphi}(t)=H_0^{-1}H_{\varphi}(t)$, then a direct calculation shows that
  \begin{equation}\label{PPTOLTHF}
|\frac{\partial}{\partial t}\log{\rm tr}h_{\varphi}(t)|\leq 4|\Phi(H_{\varphi}(t),\theta)|_{H_{\varphi}(t)},
  \end{equation}
  and
  \begin{equation}\label{PPTOLTHFI}
 |\frac{\partial}{\partial t}\log{\rm tr}h_{\varphi}^{-1}(t)|\leq 4|\Phi(H_{\varphi}(t),\theta)|_{H_{\varphi}(t)}.
 \end{equation}
 Combining Proposition \ref{PROP3_1} and the maximum principle, one can see
  \begin{equation}
  \sup\limits_{X_{\varphi}\times[0,+\infty)}|\Phi(H_{\varphi},\theta)|_{H_{\varphi}}\leq C_0.
  \end{equation}
  Integrating (\ref{PPTOLTHF}) along the time direction from $0$ to $t$, we get
  \begin{equation}
  |\log {\rm tr}h_{\varphi}(t)-\log r|=|\int_0^t\frac{\partial}{\partial s}(\log{\rm tr}h_{\varphi}(s))ds|\leq 4C_0t.
  \end{equation}
  Then we have
  \begin{equation}
  \sup\limits_{X_{\varphi}\times[0,T]}{\rm tr}h_{\varphi}(t)\leq re^{4C_0T},\quad \inf\limits_{X_{\varphi}\times[0,T]}{\rm tr}h_{\varphi}(t)\geq re^{-4C_0T},
  \end{equation}
  and
  \begin{equation}
  \sup\limits_{X_{\varphi}\times[0,T]}{\rm tr}h_{\varphi}^{-1}(t)\leq re^{4C_0T},\quad \inf\limits_{X_{\varphi}\times[0,T]}h_{\varphi}^{-1}(t)\geq re^{-4C_0T}.
  \end{equation}
  This implies that
  \begin{equation}\label{FTC_0}
    \sup\limits_{X_{\varphi}\times[0,T]}\sigma(H_0,H_{\varphi}(t))\leq 2r\left(e^{4C_0T}-1\right).
  \end{equation}
  \begin{lem}\label{Lem4_4}
    {\rm (}\protect {\cite[Lemma 6.7]{Sim88}}, \protect {\cite[Lemma 3.3]{ZZZ}}{\rm )} Suppose $u$ is a function on  $X_{\varphi}\times[0,T]$ satisfying
    \begin{equation}\label{HKOu}
    (\frac{\partial}{\partial t}-\tilde{\Delta})u\leq 0,\quad u|_{t=0}=0,
    \end{equation}
    and  there is a bound $\sup\limits_{X_{\varphi}}u\leq C_1$. Then we have
    \begin{equation}
    u(x,t)\leq\frac{C_1}{\varphi}(\varphi(x)+C_2t),
    \end{equation}
    where $C_2$ is the bound of $\tilde{\Delta}\phi$ in Assumption 2.
  \end{lem}

  In the following, we assume that there exists a constant $C_0$ such that
  \begin{equation}\label{SUPP}
    \sup\limits_{X}|\Phi(H_0,\theta)|_{H_0} \leq C_0.
  \end{equation}

 {\bf Step1:} We want to get the $C^0$-convergence of $H_{\varphi}$ on any compact set $\Omega$ in finite time. For any compact subset $\Omega\subset X$, there exists a constant $\varphi_0$ such that $\Omega\subset X_{\varphi_0}$. Let $H_{\varphi}(t)$ and $H_{\psi}(t)$ be the long-time solutions of the flow (\ref{XFAHYMF}) on $X_{\varphi}$ and $X_{\psi}$  for $\varphi_0<\varphi<\psi$. Let $u=\sigma(H_{\varphi},H_{\psi})$. Of course(\ref{FTC_0}) gives a uniform bound of $u$ and $u$ satisfies (\ref{HKOu}). By Lemma \ref{Lem4_4}, we have
 \begin{equation*}
  \sigma(H_{\varphi},H_{\psi})\leq \frac{C_1(\varphi_0+C_2T)}{\varphi_0}
  \end{equation*}
  on $X_{\varphi_0}\times[0,T].$ So $H_{\varphi} $ is a Cauchy sequence on $X_{\varphi_0}\times[0,T]$ for $\varphi\rightarrow +\infty$.

  {\bf Step2:} Get the local $C^{\infty}$-convergence of $H_{\varphi}$ on any compact subset $\Omega$. Clearly (\ref{FTC_0}) and Proposition \ref{LC_1} give the uniform $C^0$ and local $C^1$ estimate of $H_{\varphi}$. One can get the local uniform $C^{\infty}$-estimate of $H_{\varphi}$ by the standard Schauder estimate of the parabolic equation. It should be pointed out that by applying the parabolic Schauder estimate, one can only get the uniform $C^{\infty}$-estimate of $h(t)$ on $X_{\varphi}\times[\tau,T]$, where $\tau>0$ and the uniform estimate depends on $\tau^{-1}$. To fix this, one can use the maximum principle to get the local uniform bound on $|F_H|_H$, then apply the elliptic estimates to get the local uniform $C^{\infty}$-estimates. We will omit here, because it is similar to \protect \cite[Lemma 2.5]{LZZ17}. By choosing a subsequence $\varphi\rightarrow +\infty$, we know that $H_{\varphi}(t)$ converge in $C_{\rm loc}^{\infty}$-topology to a long-time solution $H(t)$ of the heat flow (\ref{AHYMF2}) on $X$.

  Let $(X,D,g,\nu)$ be a non-compact affine Gauduchon manifold satisfying the Assumptions 1 and 2, $\{X_{\varphi}\}$ an exhausting sequence of compact sub-domains of $X$.
  By the definition, it is easy to check
\begin{equation}
  |D''s|_{H_0}^2\leq \tilde{C}\langle\Psi(s)(D''s),D''s\rangle_{H_0},
\end{equation}
where $\tilde{C}$ is a positive constant depending only on the $L^{\infty}$-bound of $s$. By the identity (\ref{kiiProp}) in the compact and non-empty boundary case and the flow equation (\ref{AHYMF}), we have
\begin{equation}\label{Dlog}
\begin{split}
  \int_{X_{\varphi}}|D''s|_{H_0}^2\frac{\omega^n_g}{\nu} &\leq \tilde{C}\int_M\langle\Psi(s)(D''s),D''s\rangle_{H_0}\frac{\omega_g^n}{\nu}\\
  &=\tilde{C}\int_{X_{\varphi}}(-{\rm tr}(\Phi(H_0,\theta)s)+{\rm tr}(\Phi(H,\theta)s))\frac{\omega_g^n}{\nu}\\
  &=\tilde{C}(\int_{X_{\varphi}}(-{\rm tr}(\Phi(H_0,\theta)s))\frac{\omega_g^n}{\nu}-\frac{1}{8}\frac{d}{dt}||s||_{L^2}^2)\\
  &\leq \tilde{C}\sup\limits_{X_{\varphi}}|\Phi(H_0,\theta)|_{H_0}^2\cdot {\rm Vol}(X_{\varphi},g)\\
  &\leq C(\Phi(H_0,\theta),{\rm Vol}(X)).
\end{split}
\end{equation}
  Since the right hand side of (\ref{Dlog}) is independent of $\varphi$, we have
  \begin{equation}
      ||D''(\log(H_0^{-1}H))||_{L^2}\leq C(\Phi(H_0,\theta),{\rm Vol}(X)).
  \end{equation}
   So we obtain the following theorem.

  \begin{thm}\label{THM4_5}
    Let $(X,D,g)$ be a non-compact affine Gauduchon manifold satisfying the Assumption 1 and 2, $(E,\nabla,\theta)$ be a flat Higgs bundle with the background Hermitian metric $H_0$ over $X$. Suppose that there exists a positive number $C_0$ such that $\sup\limits_{X}|\Phi(H_0,\theta)|_{H_0}\leq C_0<+\infty$, then the affine Hermitian-Yang-Mills flow (\ref{AHYMF2}) has a long-time solution on $X\times[0,+\infty)$. Furthermore, we also have
    \begin{equation}
      ||D''(\log(H_0^{-1}H))||_{L^2}\leq C(\Phi(H_0,\theta),{\rm Vol}(X)).
    \end{equation}
 \end{thm}
 \vspace{3mm}

 Let $X$ be a non-compact  affine Gauduchon manifold satisfying the Assumption 1,2,3, and $(E,\nabla,\theta)$ be a flat Higgs bundle over $X$ with a background Hermitian metric $H_0$ satisfying $\sup\limits_{X}|{\rm tr}_gF_{H_0,\theta}|_{H_0}
<+\infty$.

  Let $\{X_{\varphi}\}_{\varphi=1}^{\infty}$ be an exhaustion sequence of compact sub-domains of $X$. On every $X_{\varphi}$, there is a unique long-time solution $H_{\varphi}(x,t)$ satisfies (\ref{XFAHYMF}).
  Due to Theorem \ref{THM4_5}, we know that $H_{\varphi}(x,t)$ converge to the long-time solution $H(x,t)$ of the affine Hermitian-Yang-Mills flow (\ref{AHYMF2}) on $X\times[0,+\infty)$ in $C_{\rm loc}^{\infty}$-topology as $\varphi\rightarrow +\infty$. We denote by $A_0$ the Chern connection on the flat vector bundle $(E, \nabla)$ with respect to the smooth  Hermitian metric $H_0$ and we sometimes denote $A(t)$ by $D_{A(t)}$. Let $(A_{\varphi}(x,t),\theta_{\varphi}(x,t))$ be the long-time solution of the affine Yang-Mills-Higgs flow on the flat vector bundle $(E, \nabla, H_0)$ over the compact affine Gauduchon manifold $(X_{\varphi},D,g)$, i.e.
 \begin{equation}\label{YMHF}
 \left\{\begin{split}
     &\frac{\partial A(t)}{\partial t}=-2(\partial_{A(t)}{\rm tr}_g-\bar{\partial}_{A(t)}{\rm tr}_g)(F_{A(t)}+[\theta(t),\theta^{*H_0}(t)]),\\
     &\frac{\partial \theta(t)}{\partial t}=-2[{\rm tr}_g(F_{A(t)}+[\theta(t),\theta^{*H_0}(t)]),\theta(t)],\\
     &A(0)=A_0,\ \  \theta(0)=\theta_0.
   \end{split}\right.
 \end{equation}
 Following Donaldson (\cite{Do85}) and Li-Zhang's (\cite{LZ11}) argument and setting $\sigma(t)^{*H_0}\circ \sigma(t)=h(t)=H_{0}^{-1}H(t)$, we can easily check that the affine Yang-Mills-Higgs flow (\ref{YMHF}) is gauge equivalent to the affine Hermitian-Yang-Mills flow (\ref{AHYMF2}) with the initial metric $H_0$. Here $\sigma(t)$ is the gauge transformation, $A(t)=\sigma(t)\circ D_{H(t)}\circ \sigma^{-1}(t)$ and $\theta(t)=\sigma(t)\circ \theta_0\circ \sigma^{-1}(t)$.

It is well known that
  \begin{equation}\label{FHT}
    |F_{H(t)}+[\theta_0,\theta_0^{*H(t)}]|^2_{H(t)}=|F_{A(t)}+[\theta(t),\theta(t)^{*H_0}]|^2_{H_0},
  \end{equation}
  and
  \begin{equation}\label{DFHT}
    \big|D_{H(t)}\big({\rm tr}_g(F_{H(t)}+[\theta_0,\theta_0^{*H(t)}])\big)\big|_{H(t)}=\big|D_{A(t)}\big({\rm tr}_g(F_{A(t)}+[\theta(t),\theta^{*H_0}(t)])\big)\big|_{H_0}.
  \end{equation}
  For simplicity, set
  \begin{equation}
    \Phi(H(t),\theta_0)={\rm tr}_g(F_{H(t)}+[\theta_0,\theta_0^{*H(t)}])
  \end{equation}
  and
  \begin{equation}
    \Phi(A(t),\theta(t))={\rm tr}_g(F_{A(t)}+[\theta(t),\theta^{*H_0}(t)]).
  \end{equation}
 Define
 \begin{equation}
   I(t)=\int_X(|D_{H(t)}\Phi(H(t),\theta_0)|_{H(t)}^2+2|[\Phi(H(t),\theta_0),\theta_0]|^2_{H(t)})\frac{\omega^n_g}{\nu}.
 \end{equation}
 Then using (\ref{FHT}) and (\ref{DFHT}) we can easily check that
  \begin{equation}
    I(t)=\int_X(|D_{A(t)}\Phi(A(t),\theta(t))|_{H_0}^2+2|[\Phi(A(t),\theta(t)),\theta(t)]|^2_{H_0})\frac{\omega^n_g}{\nu}.
  \end{equation}
 Now we want to prove that $I(t)\rightarrow 0$ as $t\rightarrow +\infty$ by using the exhaustion method. Set
\begin{equation}
I_{\varphi}(t)=\int_{X_{\varphi}}(|D_{A_{\varphi}(t)}\Phi(A_{\varphi}(t),\theta_{\varphi}(t))|_{H_0}^2+
2|[\Phi(A_{\varphi}(t),\theta_{\varphi}(t)),\theta_{\varphi}(t)]|^2_{H_0})\frac{\omega^n_g}{\nu}.
  \end{equation}
According to the  uniform bound of $\sup \limits_{X} |\Phi(H_0,\theta_0)|_{H_0}<+\infty$ and the maximum principle, we know that there exists a uniform constant $C_0$ such that
\begin{equation}
\sup\limits_{X_{\varphi}}|\Phi(A_{\varphi}(t),\theta_{\varphi}(t))|_{H_0}=\sup\limits_{X_{\varphi}}|\Phi(H_{\varphi}(t),\theta_0)|_{H_{\varphi}(t)}<\sup
  \limits_{X_{\varphi}}|\Phi(H_0,\theta_0)|_{H_0}\leq C_0.
\end{equation}
  \begin{lem}\label{Lem4_6}
    Suppose $X_{\varphi}$ is a special compact affine Gauduchon manifold with boundary $\partial X_{\varphi}$ and $||d\omega_g||_{L^{\infty}(X)}<+\infty $. If $A_{\varphi}(t)$ is a smooth solution of (\ref{YMHF}), then it holds that
    \begin{equation}\label{DILI}
      \frac{dI_{\varphi}(t)}{dt}\leq \hat{C}I_{\varphi}(t) \quad\text{for}\quad t>t_0>0,
    \end{equation}
    where $\hat{C}=\hat{C}(n,r,C_0)$ is a uniform constant independent of $t$ and $\varphi$.
  \end{lem}
  \begin{proof}
    For simplicity, we denote $A$ by $A_{\varphi}(t)$ and $\theta$ by $\theta_{\varphi}(t)$. And  we will omit $H_0$ if we know that the adjoint is with respect to $H_0$. By direct calculations, one can check
    \begin{equation}\label{PPtFA}
    \begin{split}
      \frac{\partial}{\partial t}\Phi&={\rm tr}_g(D_A(\frac{\partial A}{\partial t})+[\frac{\partial \theta}{\partial t},\theta^{*}]+[\theta,(\frac{\partial \theta}{\partial t})^{*}])\\
      &=-2{\rm tr}_gD_A(\partial_A-\bar{\partial}_A)\Phi+2{\rm tr}_g([\theta,[\Phi,\theta^{*}]]-[[\Phi,\theta],\theta^{*}]),
    \end{split}
    \end{equation}
  and
  \begin{equation}\label{ddIt}
    \begin{split}
    \frac{dI_{\varphi}(t)}{dt}&=\int_{X_{\varphi}}\frac{\partial}{\partial t}(|D_A\Phi|^2+2|[\Phi,\theta]|^2)\frac{\omega_g^n}{\nu}\\
    &=2{\rm Re}\int_{X_{\varphi}}\langle\frac{\partial}{\partial t}(D_A\Phi),D_A\Phi\rangle+2\langle\frac{\partial}{\partial t}[\Phi,\theta],[\Phi,\theta]\rangle\frac{\omega_g^n}{\nu}\\
    &=2{\rm Re}\int_{X_{\varphi}}\langle[\frac{\partial A}{\partial t},\Phi]+D_A(\frac{\partial \Phi}{\partial t}),D_A\Phi\rangle\frac{\omega_g^n}{\nu}\\
    &\quad +4{\rm Re}\int_{X_{\varphi}}\langle [\frac{\partial \Phi}{\partial t},\theta]+[\Phi,\frac{\partial \theta}{\partial t}],[\Phi,\theta]\rangle\frac{\omega_g^n}{\nu}.
   \end{split}
    \end{equation}
   By the parabolic equation and the Dirichlet boundary condition (\ref{Dirbc}), we have
   \begin{equation}
     \Phi(H_{\varphi}(t),\theta_0)\big|_{\partial X_{\varphi}}=0 \quad\text{for}\quad t>t_0>0.
   \end{equation}
 After a gauge transformation, we also see $\Phi(A_{\varphi}(t),\theta_{\varphi}(t))\big|_{\partial X_{\varphi}}=0$. So for $t>t_0>0$, we have
  \begin{equation}
    \frac{\partial}{\partial t}\Phi(A_{\varphi},\theta_{\varphi})|_{\partial X_{\varphi}}=0.
  \end{equation}
  Setting $l=[\Phi,\theta]$, we have
 \begin{equation}\label{tgmsn1}
 \begin{split}
   \langle [{\rm tr}_g[-l,\theta^*],\theta],l\rangle&=-\langle[g^{\alpha\beta}(l_{\alpha}\theta^*_{\beta}-\theta_{\beta}^*l_{\alpha}),\theta],l\rangle\\
   &=-g^{\alpha\beta}g^{\gamma\eta}{\rm tr}(((l_{\alpha}\theta_{\beta}^*-\theta_{\beta}^*l_{\alpha})\theta_{\gamma}-\theta_{\gamma}
   (l_{\alpha}\theta_{\beta}^*-\theta_{\beta}^*l_{\alpha}))l_{\eta}^*)\\
   &=-{\rm tr}((g^{\alpha\beta}(l_{\alpha}\theta_{\beta}^*-\theta_{\beta}^*l_{\alpha}))(g^{\eta\gamma}(l_{\eta}\theta_{\gamma}^*-\theta^*_{\gamma}l_
   {\eta}))^*)=-|{\rm tr}_g[l,\theta^*]|^2
\end{split}
    \end{equation}
 and
\begin{equation}\label{tfs}
 \begin{split}
   |[\Phi,\theta]|^2&=\langle l,\Phi\theta-\theta\Phi\rangle=g^{\alpha\beta}{\rm tr}(l_{\alpha}(\Phi\theta_{\beta})^*-l_{\alpha}(\theta_{\beta}\Phi)^*)\\
   &={\rm tr}(g^{\alpha\beta}(l_{\alpha}\theta_{\beta}^*-\theta_{\beta}^*l_{\alpha})\Phi^*)=\langle{\rm tr}_g[l,\theta^*],\Phi\rangle\\
   &\leq |\Phi||{\rm tr}_g[l,\theta^*]|.
 \end{split}
    \end{equation}
Applying the Stokes' formula again and substituting (\ref{PPtFA}) , we obtain
\begin{equation}\label{I_1}
\begin{split}
  &2{\rm Re}\int_{X_{\varphi}}\langle[\frac{\partial A}{\partial t},\Phi]+D_A(\frac{\partial \Phi}{\partial t}),D_A\Phi\rangle\frac{\omega^n_g}{\nu}\\
=&2{\rm Re}\int_{X_{\varphi}}(\langle[2(\bar{\partial}_A-\partial_A)\Phi,\Phi],D_A\Phi\rangle+
  \langle\frac{\partial}{\partial t}\Phi,D_A^*D_A\Phi\rangle)\frac{\omega^n_g}{\nu} \\
=&4{\rm Re}\int_{X_{\varphi}}\langle[(\bar{\partial}_A-\partial_A)\Phi,\Phi],D_A\Phi\rangle\frac{\omega^n_g}{\nu}
  -4\int_{X_{\varphi}}|D_A^*D_A\Phi|^2\frac{\omega^n_g}{\nu}\\
  & -4{\rm Re}\int_{X_{\varphi}}\langle(\tau + \bar{\tau})^*D_A\Phi,D_A^*D_A\Phi\rangle\frac{\omega^n_g}{\nu}\\
  & +4{\rm Re}\int_{X_{\varphi}}\langle{\rm tr}_g([\theta,[\Phi,\theta^*]]-[[\Phi,\theta],\theta^*]),D_A\Phi\rangle\frac{\omega^n_g}{\nu},
\end{split}
\end{equation}
where $\tau=[{\rm tr}_g,\partial\omega_g]$.

Combining (\ref{PPtFA}), (\ref{tgmsn1}) and (\ref{tfs}), we deduce
\begin{equation}\label{I_2}
\begin{split}
  &\quad 4{\rm Re}\int_{X_{\varphi}}\langle[\frac{\partial \Phi}{\partial t}, \theta],[\Phi,\theta]\rangle\frac{\omega^n_g}{\nu}\\
  &=8{\rm Re}\int_{X_{\varphi}}\langle[-D_A^*D_A\Phi,\theta],[\Phi,\theta]\rangle\frac{\omega^n_g}{\nu}-8{\rm Re}\int_{X_{\varphi}}\langle[(\tau+\bar{\tau})^*D_A\Phi,\theta],[\Phi,\theta]\rangle\frac{\omega^n_g}{\nu}\\
  &\quad +8{\rm Re}\int_{X_{\varphi}}\langle[{\rm tr}_g[\theta,(-[\Phi,\theta])^*],\theta],[\Phi,\theta]\rangle\frac{\omega^n_g}{\nu}\\
  &\quad +8 {\rm Re}\int_{X_{\varphi}}\langle[{\rm tr}_g[-[\Phi,\theta],\theta^*],\theta],[\Phi,\theta]\rangle\frac{\omega^n_g}{\nu}\\
  &\leq 16\int_{X_{\varphi}}|\theta||[\Phi,\theta]||D_A^*D_A\Phi|\frac{\omega^n_g}{\nu}+16C_1
  \int_{X_{\varphi}}|\theta||[\Phi,\theta]||D_A\Phi|\frac{\omega^n_g}{\nu}\\
  &\quad +C_2\int_{X_{\varphi}}|\theta|^2|[\Phi,\theta]|^2\frac{\omega^g}{\nu}-8\int_{X_{\varphi}}|{\rm tr}_g[l,\theta^*]|^2\frac{\omega^n_g}{\nu}
\end{split}
\end{equation}
and
\begin{equation}\label{I_3}
\begin{split}
  &\quad 4{\rm Re}\int_{X_{\varphi}}\langle[\Phi,\frac{\partial \theta}{\partial t}],[\Phi,\theta]\rangle\frac{\omega^n_g}{\nu}\\
  &=-8 {\rm Re}\int_{X_{\varphi}}\langle[\Phi,[\Phi,\theta]],[\Phi,\theta]\rangle\frac{\omega_g^n}{\nu}\\
  &\leq 16\int_{X_{\varphi}}|\Phi||[\Phi,\theta]|^2\frac{\omega^n_g}{\nu},
\end{split}
\end{equation}
where $C_1$ is a uniform constant determined by $||d\omega_g||_{L^{\infty}(X)}<+\infty$ and $C_2=C_2(n,r)$ only depends on $n$ and $r$.

Putting (\ref{I_1})-(\ref{I_3}) into (\ref{ddIt}), we have
\begin{equation}
\begin{split}
 \frac{dI_{\varphi}}{dt}\leq & C_2\int_{X_{\varphi}}(|\Phi||D_A\Phi|^2+(|\Phi|+|\theta|^2)|[\Phi,\theta]|^2+|\theta||[\Phi,\theta]|
  |D_A^*D_A\Phi|)\frac{\omega^n_g}{\nu}\\
  &-4\int_{X_{\varphi}}|D_A^*D_A\Phi|\frac{\omega^n_g}{\nu}-8\int_{X_{\varphi}}|{\rm tr}_g[l,\theta^*]|^2\frac{\omega_g^n}{\nu}\\
  & +16(C_1+1)\int_{X_{\varphi}}|\theta||[\Phi,\theta]||D_A\Phi|\frac{\omega_g^n}{\nu}
  +4C_1\int_{X_{\varphi}}|D_A\Phi||D_A^*D_A\Phi|\frac{\omega^n_g}{\nu}\\
  \leq &\hat{C}I_{\varphi}(t),
\end{split}
\end{equation}
where $\hat{C}=\hat{C}(n,r,C_0)$ is a uniform constant independent of $t$ and $\varphi$.
 \end{proof}

 \begin{thm}
   Let $H(t)$ be the long-time solution of the affine Hermitian-Yang-Mills flow (\ref{AHYMF2}) with the initial metric $H_0$, which is constructed in Theorem \ref{THM4_5}. Then $I(t)\rightarrow 0$ as $t\rightarrow +\infty$.
 \end{thm}

\begin{proof}
 By (\ref{DILI}) in Lemma \ref{Lem4_6}, we know that there exists a uniform constant $\hat{C}$ such that
 \begin{equation}\label{Iftl}
   I_{\varphi}(t)\leq e^{\hat{C}(t-s)}I_{\varphi}(s)
 \end{equation}
for any $0<t_0\leq s\leq t$.

 On the other hand, there holds that
 \begin{equation}\label{intxf}
 \begin{split}
&\int_{X_{\varphi}}|\Phi(A_{\varphi}(t),\theta_{\varphi}(t))|_{H_0}^2\frac{\omega^n_g}{\nu}+2\int_{t_0}^t\int_{X_{\varphi}}|D_{A_{\varphi}}
 \Phi(A_{\varphi},\theta_{\varphi})|_{H_0}^2\frac{\omega^n_g}{\nu}ds\\
 =&\int_{X_{\varphi}}|\Phi(A_{\varphi}(t_0),\theta_{\varphi}(t_0))|_{H_0}^2\frac{\omega^n_g}{\nu}.
\end{split}
\end{equation}
 According to Fatou's Lemma, we get
 \begin{equation}
 \begin{split}
 &\int_{X}|\Phi(A(t),\theta(t))|_{H_0}^2\frac{\omega^n_g}{\nu}+2\int_{t_0}^t\liminf\limits_{\varphi\rightarrow \infty}\int_{X_{\varphi}}|D_{A_{\varphi}}\Phi(A_{\varphi},\theta_{\varphi})|_{H_0}^2\frac{\omega^n_g}{\nu}ds\\
 \leq &\int_{X}|\Phi(A(t_0),\theta(t_0))|_{H_0}^2\frac{\omega^n_g}{\nu}.
 \end{split}
 \end{equation}
 This implies that $\int_{X}|\Phi(A(t),\theta(t))|_{H_0}^2\frac{\omega^n_g}{\nu}$  is monotonically
nonincreasing with respect to $t$. Then we have
\begin{equation}\label{intxfa}
  \bigg(\int_{X}|\Phi(A(t),\theta(t))|_{H_0}^2\frac{\omega^n_g}{\nu}-\int_{X}|\Phi(A(t+1),\theta(t+1))|_{H_0}^2\frac{\omega^n_g}{\nu}\bigg)\rightarrow 0, \qquad  \text{as}\quad t \rightarrow \infty.
\end{equation}
For any $m\geq t_0>0$, there exists $t_{m}\in [m,m+1]$ such that
\begin{equation}\label{iftm}
  I_{\varphi}(t_{m})=\int_{m}^{m+1}I_{\varphi}(t)dt.
\end{equation}
From (\ref{Iftl}),(\ref{intxf}) and (\ref{iftm}), it follows that
 \begin{equation}
 \begin{split}
  I_{\varphi}(t)\leq & e^{2\hat{C}}I_{\varphi}(t_{m})=e^{2\hat{C}}\int_m^{m+1}I_{\varphi}(t)dt\\
  =&\frac{e^{2\hat{C}}}{2}(\int_{X_{\varphi}}|\Phi(A_{\varphi}(m),\theta_{\varphi}(m))|_{H_0}^2\frac{\omega^n_g}{\nu}-\int_{X_{\varphi}}|\Phi(A_
  {\varphi}(m+1),\theta_{\varphi}(m+1))|_{H_0}^2\frac{\omega^n_g}{\nu})\\
  &+e^{2\hat{C}}\int_m^{m+1}\int_{X_{\varphi}}|[\Phi(A_{\varphi},\theta_{\varphi}),\theta_{\varphi}]|^2\frac{\omega^n_g}{\nu}dt
\end{split}
\end{equation}
for any $t\in [m+1,m+2]$. Applying Fatou's Lemma again, we drive
\begin{equation}
 \begin{split}
  I(t)\leq & \liminf\limits_{\varphi\rightarrow \infty}I_{\varphi}(t)\\
  \leq & \liminf\limits_{\varphi\rightarrow \infty}\frac{e^{2\hat{C}}}{2}(\int_{X_{\varphi}}|\Phi(A_{\varphi}(m),\theta_{\varphi}(m))|_{H_0}^2\frac{\omega^n_g}{\nu}-\int_{X_{\varphi}}|\Phi(A_
  {\varphi}(m+1),\theta_{\varphi}(m+1))|_{H_0}^2\frac{\omega^n_g}{\nu})\\
  =&\frac{e^{2\hat{C}}}{2}(\int_{X}|\Phi(A(m),\theta(m))|_{H_0}^2\frac{\omega^n_g}{\nu}-\int_{X}|\Phi(A(m+1),\theta(m+1))
  |_{H_0}^2)\\
  &+e^{2\hat{C}}\int_m^{m+1}\int_{X_{\varphi}}|[\Phi(A_{\varphi},\theta_{\varphi}),\theta_{\varphi}]|^2\frac{\omega^n_g}{\nu}dt
\end{split}
\end{equation}
for any $t\in [m+1,m+2]$. Together with (\ref{intxfa}), it means that $I(t)\rightarrow 0$ as $t\rightarrow +\infty$.
\end{proof}

\section{Stable case}

Let $(X,D,g,\nu)$ be a non-compact special affine Gauduchon manifold satisfying the Assumptions 1,2,3, and $|d\omega^{n-1}_g|_g\in L^{\infty}(X)$, $(E,\nabla,\theta)$ be a flat Higgs bundle over $X$. Fix a proper background Hermitian metric $H_0$ satisfying $\sup\limits_{X}|{\rm tr}_gF_{H_0,\theta}|_{H_0}<+\infty$ on $E$. According to \protect{\cite[Proposition 4.3]{ZZZ}}, we can solve the following Poisson equation on $(X,D,g,\nu)$:
\begin{equation}
  {\rm tr}_g\bar{\partial}\partial f=-\frac{1}{r}{\rm tr}({\rm tr}_gF_{H_0,\theta}-\lambda\cdot {\rm Id}_{E}),
\end{equation}
where
\begin{equation*}
  \lambda=\frac{\int_{X}{\rm tr}({\rm tr}_gF_{H_0,\theta})\frac{\omega^n_g}{\nu}}{{\rm rank}(E){\rm Vol}(X)}.
\end{equation*}
By a conformal change $\bar{H}_0=e^{f}H_0$, we can check that $\bar{H}_0$ satisfies
\begin{equation}\label{TTgOMC}
  {\rm tr}({\rm tr}_g(F_{\bar{H}_0}+[\theta,\theta^{*{\bar{H}_0}}])-\lambda\cdot {\rm Id}_{E})=0.
\end{equation}
So without loss of generality, we can assume that the initial metric $H_0$ satisfies the condition (\ref{TTgOMC}). Let $H(t)$ be the long-time solution of the affine Hermitian-Yang-Mills flow (\ref{AHYMF2}) with the initial metric $H_0$. Set $h(t)=H_0^{-1}H(t)=e^{s(t)}$. Then it follows that
\begin{equation}
  \log\det(h(t))={\rm tr}(s(t))=0.
\end{equation}

\begin{lem}\label{Lem5_1}
Suppose $X$ is a non-compact special affine Gauduchon manifold as before, then there are constants $C_1$ and $C_2$ such that
\begin{equation}\label{SUPST}
  \sup\limits_{X}|s(t)|\leq C_1 ||s(t)||_{L^2(X)}+C_2.
\end{equation}
\end{lem}
\begin{proof}
From  \protect{\cite[Lemma 3.1(d)]{Sim88}}, we have
  \begin{equation}
  \tilde{\Delta}\log({\rm tr}h(t)+{\rm tr}h^{-1}(t))\geq -2(|{\rm tr}_gF_{H(t),\theta}|_{H(t)}+|{\rm tr}_gF_{H_0}|_{H_0}).
  \end{equation}
  On the basis of Proposition \ref{PROP3_1}, it is easy to see that $|{\rm tr_g}F_{H(t),\theta}|_{H(t)} $ is uniformly bounded. On the other hand, we know
  \begin{equation}
    \log(\frac{1}{2r}({\rm tr}h(t)+{\rm tr}h^{-1}(t)))\leq |\log h(t)|\leq r^{\frac{1}{2}}\log({\rm tr}h(t)+{\rm tr}h^{-1}(t)).
  \end{equation}
  Then  Assumption $3$ implies (\ref{SUPST}).
\end{proof}

{\bf Proof of Theorem \ref{Maintheorem}}

When $(E,\nabla,\theta)$ is stable, we will show that, by choosing a subsequence, $H(t_i)$ converge to an affine Hermitian-Einstein metric $H_\infty$ in $C_{\text{loc}}^{\infty}$ as $t_i\rightarrow +\infty$. Clearly Proposition \ref{LC_1} and  the standard elliptic estimates mean that we only need to obtain a uniform $C^0$-estimate. By Lemma \ref{Lem5_1}, the key is to get a uniform $L^2$-estimate for $\log h(t)$, i.e.  there exists a constant $\hat{C}$ independent of $t$ such that
\begin{equation}\label{L2oLogh}
  ||\log h(t)||_{L^2}=\left(\int_X|\log h(t)|^2_{H(t)}\frac{\omega^n_g}{\nu}\right)^{\frac{1}{2}}\leq \hat{C}
\end{equation}
for all $t>0$. We prove (\ref{L2oLogh}) by contraction. If not, there would exist a subsequence $t_i\rightarrow +\infty$, such that
\begin{equation}
  \lim\limits_{i\rightarrow +\infty}||s(t_{i})||_{L^2}=+\infty
\end{equation}
and
\begin{equation}
  \varliminf_{i\rightarrow +\infty}\frac{d}{dt}||s(t)||_{L^2}\big|_{t=t_i}\geq 0.
\end{equation}
In fact, there must exist a subsequence $\hat{t}_i\rightarrow +\infty$, such that
\begin{equation}
  \lim\limits_{i\rightarrow +\infty}||s(\hat{t}_i)||_{L^2}=+\infty
\end{equation}
and
\begin{equation}
  ||s(\hat{t}_i)||_{L^2}<||s(\hat{t}_{i+1})||_{L^2}
\end{equation}
for all $\hat{t}_i<\hat{t}_{i+1}$. Denote $\tilde{t}_i=\max\{t\in [\hat{t}_i,\hat{t}_{i+1}]\,\big| \, ||s(t)||_{L^2}\leq ||s(\hat{t}_i)||_{L^2}\}$, it is easy to see that $\tilde{t_i}<\hat{t}_{i+1}$, then there exist $t_i\in[\tilde{t}_i,\hat{t}_{i+1}]$ such that
\begin{equation}
  \frac{d}{dt}||s(t)||_{L^2}\big|_{t=t_i}=\frac{||s(\hat{t}_{i+1})||_{L^2}-||s(\tilde{t}_i)||_{L^2}}{\hat{t}_{i+1}-\tilde{t}_i}
>0
\end{equation}
and $||s(t_i)||_{L^2}\geq||s(\hat{t}_i)||_{L^2}$. So we obtain a sequence $t_i \in (\hat{t}_i,\hat{t}_{i+1})$ such that $||s(t_i)||_{L^2}>||s(\hat{t}_i)||_{L^2}$ and $(\frac{d}{dt}||s(t)||_{L^2})|_{t=t_{i}}>0$.
Set
\begin{equation*}
  s(t_i)=\log h(t_i),\quad l_i=||s(t_i)||_{L^2},\quad u_i=\frac{s(t_i)}{l_i}.
\end{equation*}
Then $ {\rm tr}(u_i)=0,\quad ||u_i||_{L^2}=1$, and
\begin{equation}\label{SUPu_i}
\sup|u_i|\leq \frac{1}{l_i}(C_1l_i+C_2)<C_3<+\infty.
\end{equation}

$\bullet$ {\bf Step 1:} We show that $u_i$ converge to $u_{\infty}$  in $L_1^2$ weakly as $i\rightarrow +\infty$. We need to show that $||u_i||_{L_1^2}$ are uniformly bounded. Since $||u_i||_{L^2}=1$, it is enough to prove $||D''u_i||_{L^2}$ are uniformly bounded.

Noting (\ref{KI}) and applying to the flow equation (\ref{AHYMF2}), one can obtain
\begin{equation}
\begin{split}
  &\int_X{\rm tr}\left(\Phi(H_0,\theta)s(t)\right)+\langle\Psi(s(t))(D''s(t)),D''s(t)\rangle_{H_0}\frac{\omega^n}{\nu}\\
  =&\int_X{\rm tr}(\Phi(H,\theta)s(t))\frac{\omega^n_g}{\nu}
  =-\frac{1}{4}\int_{X}{\rm tr}(H^{-1}(t)\frac{\partial H(t)}{\partial t}s(t))\frac{\omega^n_g}{\nu}\\
  =&-\frac{1}{4}\int_X{\rm tr}(\frac{\partial s(t)}{\partial t}s(t))\frac{\omega^n_g}{\nu}
  =-\frac{1}{8}\frac{d}{dt}\int_X{\rm tr}(s^2(t))\frac{\omega^n_g}{\nu}\\
  =&-\frac{1}{8}\frac{d}{dt}||s(t)||_{L^2}^2.
\end{split}
\end{equation}
Then
\begin{equation}\label{IIne}
\int_X{\rm tr}\left(\Phi(H_0,\theta)u_i\right)+l_i\langle\Psi(l_iu_i)(D''u_i),D''u_i\rangle_{H_0}\frac{\omega^n_g}{\nu}=-\frac{1}{4}
\frac{d}{dt}||s(t_i)||_{L^2}<0.
\end{equation}
Consider the function
\begin{equation*}
  l\Psi(lx,ly)=\left\{\begin{split}
  &l,\ \ &x=y,\\
  &\frac{e^{l(y-x)}-1}{y-x},\ \ &x\neq y.
\end{split}\right.
\end{equation*}
Because of (\ref{SUPu_i}), we may assume that $(x,y)\in [C_3,C_3]\times[-C_3,C_3]$. It is easy to check that
\begin{equation}\label{LPlxly}
l\Psi(lx,ly)\rightarrow \left\{\begin{split}
  &\frac{1}{x-y},&x>y,\\
  &+\infty, &x\leq y,
\end{split}\right.
\end{equation}
increases monotonically as $l \rightarrow +\infty$. Let $\zeta \in C^{\infty}(\mathbb{R}\times \mathbb{R},\mathbb{R}^+)$ satisfy $\zeta(x,y)<\frac{1}{x-y}$ whenever $x>y$. Together with (\ref{LPlxly}), (\ref{IIne}) gives us that
\begin{equation}\label{NIIneq}
\int_X{\rm tr}\left(\Phi(H_0,\theta)u_i\right)\frac{\omega^n_g}{\nu}+\int_X\langle\zeta(u_i)(D''u_i),D''u_i\rangle_{H_0}\frac{\omega^n_g}{\nu}\leq 0,\qquad i\gg 0.
\end{equation}
In particular, take $\zeta(x,y)=\frac{1}{3C_3}$. It is obvious that when $(x,y)\in [-C_3,C_3]\times[-C_3,C_3]$ and $x>y$, $\frac{1}{3C_3}<\frac{1}{x-y}$. This implies
\begin{equation}
\int_X{\rm tr}\left(\Phi(H_0,\theta)u_i\right)\frac{\omega^n_g}{\nu}+\frac{1}{3C_3}\int_X|D''u_i|_{H_0}^2\frac{\omega^n_g}{\nu}\leq 0,
\end{equation}
for $i\gg 0$. Then
\begin{equation}
  \int_X|D''u_i|_{H_0}^2\frac{\omega^n_g}{\nu}\leq 3C_3^2\sup\limits_{X}|\Phi(H_0,\theta)|_{H_0}{\rm Vol}(X).
\end{equation}
Thus, $u_i$ are bounded in $L_1^2$. We can choose a subsequence $\{u_{i_j}\}$ such that $u_{i_j}\rightharpoonup u_{\infty}$ weakly in $L_1^2$, still denoted by $\{u_i\}_{i=1}^{\infty}$ for simplicity. Noting that $L_1^2\hookrightarrow L^2$, we have
\begin{equation}
1=\int_X|u_i|_{H_0}^2\frac{\omega^n_g}{\nu}\rightarrow \int_X|u_{\infty}|_{H_0}^2\frac{\omega^n_g}{\nu}, \quad \text{as}\quad i\rightarrow +\infty.
\end{equation}
This indicates that $||u_{\infty}||_{L^2}=1$ and $u_{\infty}$ is non-trivial. Moreover, ${\rm tr}u_{\infty}=0$. Using (\ref{NIIneq}) and following a similar discussion as that in \protect{\cite[Lemma5.4]{Sim88}}, we deduce
\begin{equation}\label{IIneouI}
  \int_X{\rm tr}\left(\Phi(H_0,\theta)u_{\infty}\right)\frac{\omega^n_g}{\nu}+\int_X\langle\zeta(u_{\infty})(D''u_{\infty}),D''u_{\infty}\rangle_{H_0}
\frac{\omega^n_g}{\nu}\leq 0.
\end{equation}

$\bullet$ {\bf Step 2} Using Uhlenbeck and Yau's trick \cite{UY86} and Loftin's argument \cite{Loftin09}, we can construct a  flat Higgs sub-bundle which contradicts the stability of $(E,\nabla,\theta)$.

By (\ref{IIneouI})  and the same argument  in \protect{\cite[Lemma 5.4]{Sim88}}, we know that the eigenvalues of $u_{\infty}$ are constant almost everywhere. Let $\mu_1<\mu_2<\cdots<\mu_l$ be the distinct eigenvalues of $u_{\infty}$. The fact that ${\rm tr}(u_{\infty})=0$ and $||u_{\infty}||_{L^2}=1$ forces $2\leq l\leq r$. For each $\mu_{\alpha}(1\leq \alpha\leq l-1)$, we construct a function $P_{\alpha}:\mathbb{R}\rightarrow\mathbb{R}$ such that
\begin{equation*}
P_{\alpha}=\left\{\begin{split}
  1,\ \ &x\leq \mu_{\alpha},\\
  0,\ \ &x\geq \mu_{\alpha +1}.
\end{split}\right.
\end{equation*}
Setting $\pi_{\alpha}=P_{\alpha}(u_{\infty})$, by \protect{\cite[Proposition 3.19]{BLS13}}, we have
\begin{enumerate}
  \item[(i)]$\pi_{\alpha}\in L_1^2$;
  \item[(ii)]$\pi_{\alpha}^2=\pi_{\alpha}=\pi_{\alpha}^{*H_0}$;
  \item[(iii)]$({\rm Id}_E-\pi_{\alpha})\bar{\partial}\pi_{\alpha}=0$;
  \item[(iv)]$({\rm Id}_E-\pi_{\alpha})[\theta, \pi_{\alpha}]=0$,
\end{enumerate}
and $\{\pi_{\alpha}\}_{\alpha=1}^{l-1}$ determine $(l-1)$ flat Higgs sub-bundles of $E$. Set $E_{\alpha}=\pi_{\alpha}(E)$. From ${\rm tr}(u_{\infty})=0$ and $u_{\infty}=\mu_l\cdot {\rm Id}_E-\sum\limits_{\alpha=1}^{l-1}(\mu_{\alpha+1}-\mu_{\alpha})\pi_{\alpha}$, it holds that
\begin{equation}\label{MLrankE}
  \mu_{l}\cdot{\rm rank}(E)=\sum\limits_{\alpha=1}^{l-1}(\mu_{\alpha+1}-\mu_{\alpha}){\rm rank}(E).
\end{equation}
Construct
\begin{equation}
\gamma=\mu_{l}{\rm deg}(E,H_0)-\sum_{\alpha=1}^{l-1}(\mu_{\alpha+1}-\mu_{\alpha}){\rm deg}(E_{\alpha},H_0).
\end{equation}
Substituting (\ref{MLrankE}) into $\gamma$, we obtain
\begin{equation}\label{gamma}
\gamma=\sum\limits_{\alpha=1}^{l-1}(\mu_{\alpha+1}-\mu_{\alpha}){\rm rank}(E_{\alpha})\left(\frac{{\rm deg}(E,H_0)}{{\rm rank}(E)}-\frac{{\rm deg}(E_{\alpha},H_0)}{{\rm rank}(E_{\alpha})}\right).
\end{equation}
Applying the same argument as that in \cite{Loftin09}, by the Chern-Weil formula, we have
\begin{equation}\label{DEGEH0}
{\rm deg}(E_{\alpha},H_0)=\frac{1}{n}\int_X({\rm tr}\left(\pi_{\alpha}{\rm tr}_gF_{H_0,\theta}\right)-|D''\pi_{\alpha}|^2)\frac{\omega^n_g}{\nu}.
\end{equation}
On the other hand, putting (\ref{DEGEH0}) into $\gamma$, we have
\begin{equation}
\begin{split}
  \gamma=&\frac{\mu_l}{n}\int_X{\rm tr}({\rm tr}_gF_{H_0,\theta})\frac{\omega^n_g}{\nu}-\frac{1}{n}\sum\limits_{\alpha=1}^{l-1}(\mu_{\alpha+1}-\mu_{\alpha})
  \int_X({\rm tr}(\pi_{\alpha}{\rm tr}_gF_{H_0,\theta})-|D''\pi_{\alpha}|^2)\frac{\omega^n_g}{\nu}\\
  =&\frac{1}{n}\int_X{\rm tr}((\mu_l\cdot{\rm Id}_E-\sum\limits_{\alpha=1}^{l-1}(\mu_{\alpha+1}-\mu_{\alpha})\pi_{\alpha})({\rm tr}_gF_{H_0,\theta}))\frac{\omega^n_g}{\nu}\\
  &+\frac{1}{n}\sum\limits_{\alpha=1}^{l-1}(\mu_{\alpha+1}-\mu_{\alpha})\int_X|D''\pi_{\alpha}|^2\frac{\omega^n_g}{\nu}\\
  =&\frac{1}{n}\int_X({\rm tr}(u_{\infty}{\rm tr}_gF_{H_0,\theta})+\langle\sum\limits_{\alpha=1}^{l-1}(\mu_{\alpha+1}-\mu_{\alpha})(dP_{\alpha})^2(u_{\infty})(D''u_{\infty}),
  D''u_{\infty}\rangle_{H_0})\frac{\omega^n_g}{\nu},
\end{split}
\end{equation}
where the function $dP_{\alpha}:\mathbb{R}\times \mathbb{R}\rightarrow \mathbb{R}$ is defined by
\begin{equation}
  dP_{\alpha}(x,y)=\left\{\begin{split}
  &\frac{P_{\alpha(x)}-P_{\alpha}(y)}{x-y},&x\neq y,\\
  &P_{\alpha}'(x),&x=y.
\end{split}\right.
\end{equation}
One can easily check that
\begin{equation}
  \sum\limits_{\alpha=1}^{l-1}(\mu_{\alpha+1}-\mu_{\alpha})(dP_{\alpha})^2(\mu_{\beta}-\mu_{\gamma})=
|\mu_{\beta}-\mu_{\gamma}|^{-1}, \quad \text{if} \quad \mu_{\beta}\neq \mu_{\gamma}.
\end{equation}
 Using (\ref{IIneouI}), we derive
\begin{equation}\label{GLTZ}
\begin{split}
  \gamma&=\frac{1}{n}\int_X({\rm tr}(u_{\infty}{\rm tr}_gF_{H_0,\theta})+\langle\sum\limits_{\alpha=1}^{l-1}(\mu_{\alpha+1}-\mu_{\alpha})(dP_{\alpha})^2(u_{\infty})(D''u_{\infty}),
  D''u_{\infty}\rangle_{H_0})\frac{\omega^n_g}{\nu}\\
  &\leq 0.
  \end{split}
\end{equation}
Then (\ref{gamma}) means
\begin{equation}
  \sum\limits_{\alpha=1}^{l-1}(\mu_{\alpha+1}-\mu_{\alpha}){\rm rank}(E_{\alpha})\left(\frac{{\rm deg}(E,H_0)}{{\rm rank}(E)}-\frac{{\rm deg}(E_{\alpha},H_0)}{{\rm rank}(E_{\alpha})}\right)\leq 0,
\end{equation}
which contradicts the stability of $E$.

\vspace{5mm}
Since we have proved that $H(t_i)$ converge to the metric $H_{\infty}$ in $C_{\rm loc}^{\infty}$ as $t_i \to +\infty$, it remains for us to show that the limit metric $H_{\infty}$ is an affine Hermitian-Einstein metric.
We will prove it by using $I(t)\rightarrow 0$ as  $t \rightarrow +\infty$.

Firstly, we can easily show that $\deg(E,H_0)=\deg(E,H_{\infty})$ under above assumptions. Indeed, set $h_\infty=H_0^{-1}H_{\infty}$, then we
\begin{equation}
 \begin{split}
  &\quad \deg(E,H_{\infty})-\deg(E,H_0)=\int_X\frac{\bar{\partial}\partial \log\det (h_\infty)\wedge \omega_g^{n-1}}{\nu}\\
  &=\int_X\frac{\bar{\partial}(\partial\log\det (h_\infty)\wedge \omega_g^{n-1})}{\nu}+
  \frac{\partial(\log\det (h_\infty)\cdot \bar{\partial}\omega_g^{n-1})}{\nu}-\frac{\log\det (h_\infty)\cdot \partial\bar{\partial}\omega_g^{n-1}}{\nu}\\
  &=(-1)^{n-1}\int_Xd(\frac{\partial \log\det (h_\infty)\wedge \omega_g^{n-1}}{2\nu})+\int_Xd(\frac{\log\det (h_\infty)\cdot \bar{\partial}\omega_g^{n-1}}{2\nu})\\
  &=0,
\end{split}
\end{equation}
where we have used Lemma \ref{SFLem} to show that the last two terms vanish. The first term vanishes because  $||D''(\log(H_0^{-1}H(t)))||_{L^2}$ is uniformly bounded and the second vanishes because $|d\omega_g^{n-1}|\in L^{\infty}(X)$ and $|\log\det (h_\infty)|\in L^{\infty}(X)$.

From $I(t)\rightarrow 0$ as $t\rightarrow +\infty$, we know that
\begin{equation}
\left\{\begin{split}
    D_{H_{\infty}}\Phi_{\infty}=0,\\
    [\Phi_{\infty},\theta_{\infty}]=0,
  \end{split}\right.
\end{equation}
where $\Phi_{\infty}=\Phi(H_{\infty},\theta_{\infty})$. Since $\Phi_{\infty}$ is parallel and $\Phi_{\infty}^{*H_{\infty}}=\Phi_{\infty}$, we can decompose $(E_{\infty},H_{\infty},\theta_{\infty})$ according to the real eigenvalues $\lambda_1\cdots \lambda_{k}$ of $\Phi_{\infty}$, i.e.
\begin{equation}
  (E_{\infty},H_{\infty},\theta_{\infty})=\bigoplus_{i=1}^k (E^i_{\infty},H_{\infty}|_{E_{\infty}^i},\theta_{\infty}|_{E^i_{\infty}}).
\end{equation}
Set $H^i=H_{\infty}|_{E^i_{\infty}}$ and $\Phi_i=\Phi_{\infty}|E^i_{\infty}$, then we have
\begin{equation}
  \Phi_i=\lambda_i {\rm Id}_{E^i_{\infty}}.
\end{equation}
If $k=1$, we are done. Otherwise,  it contradicts the $H_0$-stability of $E$. In fact, since $E^i_{\infty}$ is mutually orthogonal flat sub-bundle of $E$ with respect to $H_{\infty}$, we have
\begin{equation}
  \deg(E,H_0)=\sum\limits_{i=1}^{k}\deg(E^i_{\infty},H^i).
\end{equation}
So there exists  an $i_0$ such that $\mu_g(E,H_0)\leq \mu_g(E_{\infty}^{i_0},H^{i_0})$. Then a contradiction occurs.

\vskip 1 true cm

%-----------------------------------------------------------------------------
%-----------------------------------------------------------------------------

\end{document}